\documentclass[reqno]{amsart}
\usepackage{amsmath}
\usepackage[normalem]{ulem}
\usepackage{cancel}
\usepackage{extarrows}
\usepackage{mathrsfs}
\usepackage{xcolor}
\usepackage
[
a4paper,% other options: a3paper, a5paper, etc
vmargin=3.1cm,
hmargin=2.5cm
]{geometry}
\usepackage{latexsym,amsmath,amscd,amssymb,amsthm,graphics,float,enumerate}
\usepackage{algorithm, algpseudocode}
\usepackage[shortlabels]{enumitem}
\usepackage{graphicx}
\usepackage{mathtools}
\usepackage[colorlinks]{hyperref}
\usepackage{url}
\usepackage{framed}
\usepackage[all]{xy}
\usepackage{epstopdf}
\usepackage{subcaption}
\usepackage[titletoc,title]{appendix}
\usepackage[foot]{amsaddr}
\newtheorem{theorem}{Theorem}[section]
\newtheorem{definition}[theorem]{Definition}

\newtheorem{remark}[theorem]{Remark}
\newtheorem{proposition}[theorem]{Proposition}

\newtheorem{assump}{Assumption}

\renewcommand{\vec}[1]{\boldsymbol{#1}}

\def\diff{\mathrm{d}}

\def\R{\mathbb R}

\def\L{\mathcal L}

\begin{document}
\title{A Geometric Extension of the It\^o-Wentzell and Kunita's Formulas}

\author[A. Bethencourt de Le\'{o}n]{Aythami Bethencourt de Le\'{o}n $^{1,}$$^{2,}$$^*$}
\address{$^{1}$ Department of Mathematics, Imperial College London, London SW7 2AZ, UK. \href{mailto:ab1113@ic.ac.uk}{ab1113@ic.ac.uk}}

\address{$^{2}$ Department of Fundamental Mathematics, University of La Laguna, San Cristóbal de La Laguna 38206, Spain.}

\author[S. Takao]{So Takao$^{1,}$$^3$}

\address{$^{3}$ UCL Centre for Artificial Intelligence, University College London, London WC1V 6BH, UK.}

\address{$^{*}$ Corresponding author.}

\maketitle

\begin{abstract}
We extend the It\^o-Wentzell formula for the evolution along a continuous semimartingale of a time-dependent stochastic field driven by a continuous semimartingale to tensor field-valued stochastic processes on manifolds. More concretely, we investigate how the pull-back (respectively, the push-forward) by a stochastic flow of diffeomorphisms of a time-dependent stochastic tensor field driven by a continuous semimartingale evolves with time, deriving it under suitable regularity conditions. We call this result the Kunita-Itô-Wentzell (KIW) formula for the advection of tensor-valued stochastic processes. Equations of this nature bear significance in stochastic fluid dynamics and well-posedness by noise problems, facilitating the development of certain geometric extensions within existing theories.
% This has some interesting potential applications to stochastic fluid dynamics and well-posedness by noise problems.

%Our results are innovative because we present an example of well-posedness by noise for a large geometric class of equations.

%\textcolor{red}{In the abstract, I would stress the fact that in the literature, well-posedness by noise is limited to a few equations (transport, continuity,...), and here, one of the purposes is establishing well-posedness by noise for a ``large'' group of equations, understanding that well-posedness by noise is not merely a feature of linear transport but rather a much more general phenomenon that has a geometric origin. The abstract has to be powerful}
% with initial data $K_0$ of class $L^p \cap \L^\infty_{loc},$ $1 \leq p < \infty,$ and 
% drift and diffusion vector field satisfying Assumption \eqref{assump}.
% This is an example of a general class of equations admitting well-posedness by addition of noise, since we can show that \eqref{eq0} allows multiple solutions for the same initial profile in the deterministic case, and is therefore ill-posed.
% Our formulation and results provide a general geometric viewpoint for deriving well-posedness results for the scalar transport equation, continuity equation, vector advection equation, and other particular cases, and hence permits the generalisation of popular results in the literature of well-posedness by stochastic noise.

\noindent{{\bf Key words:} It\^o-Wentzell formula; stochastic processes on manifolds; tensor-valued semimartingales; Lie transport equation.}
\end{abstract}
\newpage

%\tableofcontents

%\newpage

\section{Introduction}
The purpose of this paper is to derive a geometric extension of the classical It\^o's Lemma with appropriate regularity conditions, which also extends the It\^o-Wentzell formula \cite{kunita1981some} and geometric It\^o-type formulas established by Kunita in \cite{kunita1997stochastic}. Such formulas can be understood simply as a stochastic extension of the ``chain rule" in multivariable calculus or differential geometry. More concretely, in this paper we establish formulas for the pull-back and push-forward of tensor-valued stochastic processes along stochastic flows of diffeomorphisms, which we call the {\em Kunita-It\^o-Wentzell (KIW) formulas for the advection of tensor-valued stochastic processes}.
We believe that these results are fundamental in stochastic analysis on manifolds. However, to our knowledge, such results are not available in the literature.

%In particular, we present a strong and a weak version of this result. 

%for an It\^o stochastic process, this is, taking the form
%\begin{align*}
%    & \diff X_t = \mu_t \, \diff t + \sigma_t \, \diff \diff B_t,
%\end{align*}
%and any twice differentiable scalar function of two real variables, one has
%\begin{align*}
%    & \diff f(t,X_t) =\left(\frac{\partial f}{\partial t} + \mu_t \frac{\partial f}{\partial x} + \frac{\sigma_t^2}{2}\frac{\partial^2f}{\partial x^2}\right) \diff t+ \sigma_t \frac{\partial f}{\partial x}\, \diff B_t.  
%\end{align*}
%This immediately implies that $f(t,x_t)$ is itself an Itô drift-diffusion process. In higher dimensions, 
Let $T>0.$ It\^o's Lemma in its simplest version states that if $f:\mathbb{R}^n \rightarrow \mathbb{R}$ is sufficiently regular and $\vec{X}_t$ is an $\mathbb{R}^n$-valued It\^o process, hence described by the equation
\begin{align*}
    \vec{X}_t = \vec{X}_0 + \int_0^t \vec{\mu}_s \, \diff s + \int_0^t \vec{\sigma}_s \, \diff W_s, \quad t \in [0,T],
\end{align*}
where $\vec{\mu}, \vec{\sigma}:\Omega \times [0,T] \rightarrow \mathbb{R}^n$ are progressively measurable processes and $W_t$ is Brownian motion, then $\mathbb{P}$-a.s., $f(\vec{X}_t)$ satisfies the equation
\begin{align}
     f(\vec{X}_t) &= f(\vec{X}_0) + \sum_{j=1}^n \int_0^t \frac{\partial f}{\partial x_j} (\vec{X}_s) \, \diff X_s^j  + \frac{1}{2} \sum_{i,j=1}^n \int_0^t  \frac{\partial^2 f}{\partial x_i \partial x_j} (\vec{X}_s) \, \diff [X^i_\cdot ,X^j_\cdot]_s \nonumber \\
     &= f(\vec{X}_0) + \int_0^t \left( \left (D f(\vec{X}_s) \right)^T \vec{\mu}_s + \frac{1}{2}\operatorname{Tr}\left[ \vec{\sigma}_s^T \left( D^2 f(\vec{X}_s) \right) \vec{\sigma}_s \right] \right) \, \diff s + \int_0^t \left (D f(\vec{X}_s) \right)^T \vec{\sigma}_s\, \diff W_s, \label{eq:og-ito-formula}
\end{align}
for $t \in [0,T],$ where %$D^2$ stands for the Hessian of a function, and 
$\operatorname{Tr}$ is the trace of a matrix. It\^o's Lemma is the most basic of stochastic chain rules and is fundamental, among other things, as it tells us that if $\vec{X}_t$ is an It\^o process and $f$ is sufficiently regular, then $f(\vec{X}_t)$ is also an It\^o process. In other words, the relation $\vec{X}_t \rightarrow f(\vec{X}_t)$ maps the space of It\^o processes into itself. It is well known that It\^o's Lemma can be extended to stochastic processes driven by continuous semimartingales $S_t,$ i.e. of the form 
\begin{align*}
    \vec{X}_t = \int_0^t \vec{\mu}_s \, \diff S_s = \int_0^t \vec{\mu}_s \, \diff A_s + \int_0^t \vec{\mu}_s \, \diff M_s, \quad t \in [0,T],
\end{align*}
where we have expressed $S_t = A_t + M_t$ uniquely as the sum of a continuous process of finite variation $A_t$ with $A_0=0$, $M_t$ is a continuous local martingale, and $\vec{\mu}_t$ is a progressively measurable process. In this general case, it can be expressed more compactly as
\begin{align*} 
    & f(\vec{X}_t) =  f(\vec{X}_0) + \sum_{j=1}^n \int_0^t \frac{\partial f}{\partial x_j} (\vec{X}_s) \, \diff X_s^j  + \frac{1}{2} \sum_{i,j=1}^n \int_0^t  \frac{\partial^2 f}{\partial x_i \partial x_j}(\vec{X}_s) \, \diff [X^i_\cdot ,X^j_\cdot]_s, \quad t \in [0,T].
\end{align*}
Again, among other things, It\^o's Lemma for semimartingales tells us that $\vec{X}_t \rightarrow f(\vec{X}_t)$ maps the space of continuous semimartingales into itself. Versions of It\^o's Lemma are available in the literature even for non-continuous semimartingales \cite{michel,protter} and in Banach spaces \cite{rozovskii}, but here we will focus on the continuous case. 

In this article, we prove an extension of It\^o's Lemma allowing for {\em all} of the following features: 
\begin{enumerate}[label=(\roman*)]
\item Instead of $f$ being deterministic, it can be a stochastic process $f_t,$ $t \in [0,T]$.
\item Instead of $f$ being scalar-valued, $f$ can be tensor-valued, and further defined over a manifold $Q$.
\item $X_t = \phi_t(x)$ is a stochastic flow of diffeomorphisms, in particular the flow of an SDE, and $X_t$ acts on $f$ by pull-back/push-forward, i.e. $\phi_t^* f$ or $(\phi_t)_* f$. These naturally generalise the notion of ``evaluation" on the flow/inverse flow to tensor fields.
% Advection becomes interesting when $f$ is not scalar-valued, since for scalar-valued functions, it becomes is evaluation on the flow $\phi_t^*f=f(\phi_t)$ as in the classical It\^o formulas.
\end{enumerate}
An extension of It\^o's Lemma that includes (i) is known in the mathematical community as the It\^o-Wentzell formula \cite{kunita1981some}.
% In \cite{kunita1981some}, Kunita added (i) to It\^o's Lemma \textcolor{blue}{ST: wording a bit strange}, which is known in the mathematical community as the It\^o-Wentzell formula.
The It\^o-Wentzell formula states that given a sufficiently regular semimartingale field $\vec{f} : \Omega \times [0,T] \times \mathbb{R}^n \rightarrow \mathbb{R}^n$ of the form
\begin{align*}
\vec{f}(t,\vec{x}) = \vec{f}(0,\vec{x}) +  \int_0^t \vec{g}(s,\vec{x}) \, \diff S_s, \quad t \in [0,T], \quad \vec{x} \in \mathbb{R}^n,
\end{align*}
where $\vec{g}: \Omega \times [0,T] \times \mathbb{R}^n \rightarrow \mathbb{R}^n$ is a sufficiently regular adapted process and $S:\Omega \times [0,T] \rightarrow \mathbb{R}$ is a continuous semimartingale, then for a continuous semimartingale $\vec{M}: \Omega \times [0,T] \rightarrow \mathbb{R}^n,$ $\vec{f}(t,\vec{M}_t)$ satisfies the following relation
\begin{align*} 
    & \vec{f}(t,\vec{M}_t) =  \vec{f}(0,\vec{M}_0) + \int^t_0 \vec{g}(s,\vec{M}_s) \, \diff S_s  + \sum_{j=1}^n \int_0^t \frac{\partial \vec{f}}{\partial x_j} (s, \vec{M}_s) \, \diff M_s^j \nonumber \\
    & + \sum_{j=1}^n \int_0^t  \frac{\partial \vec{g}}{\partial x_j} (s, \vec{M}_s) \, \diff [S_\cdot ,M^j_\cdot]_s + \frac{1}{2} \sum_{i,j=1}^n \int_0^t  \frac{\partial^2 \vec{f}}{\partial x_i \partial x_j} (s, \vec{M}_s) \, \diff [M^i_\cdot ,M^j_\cdot]_s,\quad t \in [0,T].
\end{align*}
The It\^o-Wentzell formula is a useful chain rule that as a corollary tells us that $\vec{M}_t \rightarrow \vec{f}(t, \vec{M}_t)$ maps the space of continuous semimartingales into itself. Extensions of the It\^o-Wentzell formula in low regularity conditions can be found, for instance, in \cite{krylov2011ito}, where the author employs mollifiers to derive a version for distribution-valued processes.  We remark that there exists another classical variation of the It\^o-Wentzell formula, established by Bismut in \cite{itowentzell1} at the same time Kunita proved his formula, where the author computes the evaluation of a stochastic flow along a continuous semimartingale (and not the other way around as we do in Theorem \ref{itoflows} %and \ref{straflows}, 
which is a completely different problem). Adding point (iii) to the It\^o-Wentzell formula can be easily done as a corollary since pull-back by a flow $\vec{\phi}_t$ of a scalar-valued function $f(t,\vec{x})$ on $\mathbb{R}^n$ is given by $\vec{\phi}_t^* f(t,\vec{x}) = f(t,\vec{\phi_t(x)}),$ so the standard It\^o-Wentzell formula can be applied. We present this result in the preliminaries (Section \ref{sec:preliminaries}).
Moreover, extensions of It\^o's Lemma containing (ii)+(iii) (without
including point (i)) were also established by Kunita in \cite{kunita1981some,kunita1997stochastic}. Indeed, if $K$ is a time-independent sufficiently regular $(r,s)$-tensor field defined on a manifold $Q,$ and $\phi_t: Q \rightarrow Q,$ $t \in [0,T]$ is the flow of the SDE 
\begin{align} \label{phi-map-intro}
\diff \phi_{t}(x) = b(t,\phi_{t}(x)) \, \diff t + \xi(t,\phi_{t}(x)) \circ \diff B_t, \quad t \in [0,T], \quad x \in Q,
\end{align}
where $b,\xi:[0,T] \times Q \rightarrow TQ$ are sufficiently regular vector fields, $B_t$ is Brownian motion (we have changed from the notation $W_t$ for convenience), and $\circ$ represents Stratonovich integration \footnote{We write SDE \eqref{phi-map-intro} in Stratonovich form as Kunita does in his books. We note that in order to derive formulas for geometric objects of the type $\phi_t^* K$, it is more convenient to write it in Stratonovich form. Since we assume enough regularity, the respective formula for the SDE in It\^o form can be easily calculated.}, then the tensor-valued stochastic process $\phi_{t}^* K$ satisfies It\^o's second formula
\begin{align} 
&\phi_{t}^* K = K + \int_0^t \phi_{s}^* (\mathcal{L}_{b} K) \, \diff s
+  \int_0^t \phi_{s}^* (\mathcal{L}_{\xi} K) \, \diff B_s -  \frac{1}{2} \int_0^t \phi_{s}^* (\mathcal{L}^2_{\xi} K) \, \diff s, \quad t \in [0,T],
\label{ito-second-1}
\end{align}
whereas $(\phi_t)_* K$ satisfies It\^o's first formula
\begin{align} 
&(\phi_{t})_* K = K + \int_0^t  \mathcal{L}_{b} [(\phi_{s})_* K] \, \diff s
+  \int_0^t \mathcal{L}_{\xi} [(\phi_{s})_* K] \, \diff B_s -  \frac{1}{2} \int_0^t \mathcal{L}_{\xi}^2 [\phi_{s}^* K] \, \diff s, \quad t \in [0,T].
\label{ito-first-1}
\end{align}
In \cite{kunita1997stochastic}, Kunita works with slightly more general flows depending on two time parameters $\phi_{s,t},$ $0\leq s \leq t \leq T$, corresponding to the initial and final times, and provides four similar formulas depending on which parameter is fixed when performing the integration.

In the present paper, we derive an extension of It\^o's Lemma allowing for all (i),(ii) and (iii), which we call the Kunita-It\^o-Wentzell (KIW) formula for the advection of a tensor-valued stochastic process. This is stated in both a version in It\^o and Stratonovich form, under different regularity assumptions. We stress that our theorem extends {\em all the previous results} and is valid on quite general manifolds. Without being precise about the regularity conditions along the lines of this introduction (we will state this more rigorously in Section \ref{sec:main-result}), the statement of our main theorem is the following.

\begin{theorem}[Kunita-It\^o-Wentzell (KIW) formula for the advection of tensor-valued processes] Let $Q$ be a smooth manifold and $K: \Omega \times [0,T] \rightarrow  \Gamma(T^{(r,s)}Q)$ be a sufficiently smooth tensor-valued semimartingale. Moreover, let $S^i_t = A^i_t + M^i_t$ be continuous semimartingales expressed uniquely as the sum of a continuous process of bounded variation with $A_0^i=0$ and a continuous local martingale, $t \in [0,T],$ $i=1,\ldots,M.$ Assume that $K$ has the explicit form
\begin{align*} 
    K(t,x) = K(0,x) + \sum_{i=1}^M \int^t_0 G_i(s,x) \, \diff A^i_s + \sum_{i=1}^M  \int^t_0 G_i(s,x) \, \diff M^i_s, \quad t \in [0,T],\hspace{0.2cm} x \in Q,
\end{align*}
where $G_i$ are smooth enough adapted tensor-valued processes for $i=1,\ldots,M$. Let $\phi_t$ be the semimartingale solution to the SDE \eqref{flow-map} with smooth enough drift and diffusion vector fields. Then the tensor-valued process $K$ pulled back by the flow $\phi_t$ satisfies the following relation 
\begin{align} \label{KIWkformsimplified}
\begin{split}
    &\phi_t^* K(t,x) =  K(0,x) + \sum_{i=1}^M \int^t_0 \phi_s^* G_i(s,x) \, \diff A^i_s + \sum_{i=1}^M \int^t_0 \phi_s^* G_i(s,x) \,  \diff M^i_s  \\
    & \hspace{50pt}+ \int^t_0 \phi_s^* \mathcal L_b K (s,x) \,  \diff s 
    + \sum_{j=1}^N \int^t_0 \phi_s^* \mathcal L_{\xi_j} K(s,x) \, \diff B_s^j  \\
    &\hspace{20pt}+ \sum_{i=1}^M \sum_{j=1}^N \int^t_0 \phi^*_s \mathcal L_{\xi_j} G_i(s,x) \,  \diff [M^i_{\cdot}, B^j_{\cdot}]_s + \frac12 \sum_{j=1}^N  \int^t_0 \phi^*_s \mathcal L_{\xi_j} \mathcal L_{\xi_j} K(s,x) \,  \diff s,
\end{split}
\end{align}
for all $t \in [0,T]$ and $x \in Q.$ Moreover, the following formula holds for the push-forward of $K$ by the flow
\begin{align} \label{KIWkformsimplifiedpush}
\begin{split}
    &(\phi_t)_* K(t,x) =  K(0,x) + \sum_{i=1}^M \int^t_0 (\phi_s)_* G_i(s,x) \, \diff A^i_s + \sum_{i=1}^M \int^t_0 (\phi_s)_* G_i(s,x) \,  \diff M^i_s  \\
    & \hspace{50pt} - \int^t_0 \mathcal L_b [(\phi_s)_* K](s,x) \,  \diff s 
    - \sum_{j=1}^N \int^t_0 \mathcal L_{\xi_j} [(\phi_s)_* K](s,x) \, \diff B_s^j \\
    &\hspace{20pt} - \sum_{i=1}^M \sum_{j=1}^N \int^t_0 \mathcal L_{\xi_j} [(\phi_s)_* G_i](s,x) \,  \diff [M^i_{\cdot}, B^j_{\cdot}]_s + \frac12 \sum_{j=1}^N  \int^t_0 \mathcal L_{\xi_j} \mathcal L_{\xi_j} [(\phi_s)_* K](s,x) \,  \diff s.
\end{split}
\end{align}
\end{theorem}
\begin{remark}
As we will see, \eqref{KIWkformsimplifiedpush} requires more regularity on the coefficients of the SDE \eqref{flow-map} than \eqref{KIWkformsimplified}.
\end{remark}

We can also derive a version in Stratonovich form under stronger regularity assumptions. We state this below, again without discussing the precise regularity hypotheses.
\begin{theorem}[KIW formula for the advection of tensor-valued processes: Stratonovich version] 
Let $Q$ be a smooth manifold. Let $K: \Omega \times [0,T] \rightarrow  \Gamma(T^{(r,s)}Q)$ be a sufficiently smooth tensor-valued semimartingale. Moreover, let $S^i_t = A^i_t + M^i_t$ be continuous semimartingales expressed uniquely as the sum of a continuous process of bounded variation with $A_0^i=0$ and a continuous local martingale, $t \in [0,T],$ $i=1,\ldots,M.$ Let $K$ have the explicit form
\begin{align*}
    K(t,x) = K(0,x) + \sum_{i=1}^M \int^t_0 G_i(s,x) \, \diff A_s^i + \sum_{i=1}^M \int^t_0 G_i(s,x) \circ \diff M_s^i, \quad t \in [0,T],\quad x \in Q,
\end{align*}
where $G_i$ are smooth enough tensor-valued adapted processes,
$i=1,\ldots,M$. Let $\phi_t$ be the semimartingale solution to the SDE \eqref{flow-map} with smooth enough drift and diffusion vector fields. Then the tensor-valued process $K$ pulled back by the flow $\phi_t$ satisfies the following relation
\begin{align*} 
    \phi_t^* K(t,x) &=  K(0,x) + \sum_{i=1}^M \int^t_0 \phi_s^* G_i(s,x) \, \diff A^i_s + \sum_{i=1}^M \int^t_0 \phi_s^* G_i(s,x) \circ \diff M_s^i \nonumber \\
    &+ \int^t_0 \phi_s^* \mathcal L_b K (s,x) \, \diff s 
    + \sum_{j=1}^N \int^t_0 \phi_s^* \mathcal L_{\xi_j} K(s,x) \circ \diff B_s^j,\quad 
\end{align*}
for all $t \in [0,T]$ and $x \in Q.$ We also have the following formula for the push-forward of $K$ by the flow
\begin{align*} 
    (\phi_t)_* K(t,x) &=  K(0,x) + \sum_{i=1}^M \int^t_0 (\phi_s)_* G_i(s,x) \, \diff A_s^i + \sum_{i=1}^M \int^t_0 (\phi_s)_* G_i(s,x) \circ \diff M_s^i \nonumber \\
    & - \int^t_0 \mathcal L_b [(\phi_s)_* K] (s,x) \, \diff s
    - \sum_{j=1}^N \int^t_0 \mathcal L_{\xi_j} [(\phi_s)_* K](s,x) \circ \diff B_s^j.
\end{align*}
\end{theorem}
%Formulas \eqref{spde-compact} and \eqref{KIWkformsimplified} are compact forms of equations \eqref{K-eq-strat} and \eqref{Ito-Wentzell-one-form-Strat-ver} in Section \ref{sec-KIW-for-k-forms}. The latter equations are written in integral notation to make the stochastic processes more explicit.
%\begin{remark}\rm
%In applications, one can formally express \eqref{KIWkformsimplified} using the differential notation
%\begin{align*}
%\diff \,(\phi_t^* K)(t,x) = \phi_t^*\left(\diff K + \mathcal L_{\diff X_t} K\right)(t,x),
%\end{align*}
%where $\diff X_t$ is the stochastic vector field $\diff X_t(x) = b(t,x) \diff t + \sum_{i=1}^N \xi_i(t,x) \circ \diff B_t^i$.
%\end{remark}
We note that in the case $Q:=\mathbb{R}^n,$ $S_t^1:=t,$ $S_t^i:=W_t^i,$ $i=2,\ldots,M,$ and  $K$ being a $k$-form, the KIW formula for the pull-back was established in \cite{takao2019} by means of a mollifying approach, motivated by the techniques employed in \cite{krylov2011ito} to prove the It\^o-Wentzell formula for distribution-valued processes.

In a related work, \cite{crisan2022variational} proves a version of KIW on tensor fields over manifolds in the general case where the driving process is given by a geometric rough path. However, in their work, they assume $C^\infty$-regularity on all the coefficients, enabling a simplified argument. Moreover, in \cite{catuogno2022geometric}, a  version of KIW on manifolds with rough driving processes is proven under suitable regularity assumptions, but only on trivial vector bundles (i.e., a tuple of scalar fields).
In the present work, we obtain results in the case of general tensor fields with semimartingale driving processes under appropriate regularity conditions, leading to proofs that require a delicate balancing of regularities and a careful treatment of stochasticity. In particular, we demonstrate that the regularities needed to establish the formulas for the pull-back \eqref{KIWkformsimplified} differ from that of the push-forward \eqref{KIWkformsimplifiedpush}, which is not seen in the scalar case. 
Another distinction with the afformentioned line of works is that geometric rough paths only generalise the notion of Stratonovich noise.
We demonstrate that within the Itô formulation, we can establish our results under more refined regularity assumptions compared to the Stratonovich formulation. This subtle nuance remains hidden when working exclusively with geometric rough paths.

This work is part of a larger investigation of stochastic processes defined on manifolds, initiated by the significant works of Kunita \cite{kunita1997stochastic}, Bismut \cite{bismut1982mecanique}, Emery \cite{emery2012stochastic}, Elworthy-Li-LeJan \cite{elworthy1998stochastic, elworthy2007geometry, elworthy2010geometry} and others. Here, our focus is on the interplay of stochastic flows and tensor field valued-semimartingales on manifolds, which have motivations in stochastic geometric fluid mechanics \cite{holm2015variational, takao2019} and corresponding regularisation-by-noise properties, as investigated in \cite{kforms}.

% However, the mathematical approach in the present paper is fully rigorous, whereas in \cite{takao2019}, some mathematical details are treated formally and focus is made on variational principles and physical applications. \textcolor{blue}{ST: do we really need this?}
%A quick comparison of the It\^o--Wentzell formulas in \eqref{Ito-Wentzell-intro} and \eqref{KIWkformsimplified} shows the parallels and differences between the scalar and tensor cases. Our proof of this theorem relies on a slight extension of Krylov's mollifier approach in \cite{krylov2011ito}. Our proof uses mollifiers to evaluate the tensor-valued process $K(t,x)$ along the flow $\phi_t$ without having to discretise the time and take limits, as is usually done.
%The result for deterministic, smooth-in-time $K$ is already available in \cite{kunita1984stochastic},
%The Kunita It\^o-Wentzell Theorem
%proved in the present paper enhances the It\^o-Wentzell formula of \cite{krylov2011ito}, by generalising the stochastic scalar function $f$ to a $k$-form-valued process.
%and for the particular case in which $K$ is a deterministic tensor-valued process, some consequences in fluid dynamics have also been discussed previously in \cite{catuogno2016stochastic,rezakhanlou2016stochastically}. 

\section{Preliminaries} \label{sec:preliminaries}
We introduce some preliminary concepts and results that will be employed in this article. In particular, in Subsection \ref{holder}, we discuss H\"older spaces on the Euclidean space, in Subsection \ref{back-banach}, we define an extended notion of H\"older spaces to tensor fields over manifolds, in Subsection \ref{sec:tensor-operations} we introduce natural operations defined on the space of tensor fields, namely, pull-back, pushforward and Lie derivative, which we use frequently in this work, and in Subsection \ref{probabilidad}, we introduce some classical results from the theory of stochastic processes. Throughout the text, we follow the convention of using {\bf bold font} to denote vectors/tensors in the Euclidean space (including local coordinate expressions of variables on manifolds) and in {\it italics} otherwise.

\subsection{H\"older spaces on Euclidean domains} \label{holder}
Let $\alpha \in (0,1).$ We say that a Borel-measurable function $f:\mathbb{R}^n \rightarrow \mathbb{R}$ is (locally) $\alpha$-H\"older continuous if for any compact subset $W \subset \mathbb{R}^n$, we have
\begin{align} \label{holder-seminorm}
[f]_{\alpha, W} := \sup_{\vec{x}, \vec{y} \in W} \frac{|f(\vec{x})-f(\vec{y})|}{|\vec{x}-\vec{y}|^{\alpha }}<\infty.
\end{align}
We denote the space of $\alpha$-H\"older continuous functions by $C^\alpha(\mathbb{R}^n),$ which is a Fr\'echet space endowed with the seminorm \eqref{holder-seminorm}.
% \begin{align*}
%     \norm{f}_{C_b^\alpha} := \norm{f}_{L^\infty} + [f]_{\alpha}.
% \end{align*}
This definition can be straightforwardly extended to functions of multiple components, in which case we employ the notation $C^\alpha(\mathbb{R}^n,\mathbb{R}^m).$ For $k\in \mathbb{N},$ we define $C^{k+\alpha}(\mathbb{R}^n)$ to be the space of functions $f:\mathbb{R}^n \rightarrow \mathbb{R}$ having continuous derivatives of order $i=0,\ldots,k,$ and $k$-th derivatives of class $C^\alpha(\mathbb{R}
^n).$
% The corresponding norm is defined as
% \begin{align*}
% \norm{f}_{C_{b}^{k+\alpha}} := \sum_{i=0}^{k}\norm{D^{i}f}_{L^\infty}+[D^{k}f]_{\alpha},
% \end{align*}
% which makes it a Banach space.
Again, this definition can easily be extended to vector-valued functions. As before, $C^{k +\alpha}(\mathbb{R}^n, \mathbb{R}^n)$ defines a Fr\'echet space, whose corresponding seminorm is given by
\begin{align}
[\vec{f}]_{k+\alpha, W} := \sup_{\vec{x} \in W} |\vec{f}(\vec{x})| + \sum_{1 \leq |\vec{v}| < k} \sup_{\vec{x} \in W} |D^{\vec{v}} \vec{f}(\vec{x})| + \sum_{|\vec{v}| = k} \sup_{\vec{x}, \vec{y} \in W} \frac{|D^{\vec{v}} \vec{f}(\vec{x})- D^{\vec{v}}\vec{f}(\vec{y})|}{|\vec{x}-\vec{y}|^{\alpha }}.
\end{align}

\subsection{H\"older sections of tensor bundles over manifolds} \label{back-banach}
% In this section, we define integrability and differentiability classes of tensor fields on manifolds. For further details and material on this topic, we point the reader to \cite{bauer2010sobolev,bauer2013sobolev,dodziuk1981sobolev,abra}). % We also explain how to mollify tensors fields on Riemannian manifolds by applying the techniques in \cite{mollifier} and define some convenient weak notions of the Lie derivative. 
%\textcolor{red}{A: Include references as to where this material can be found in the literature}
%We only consider tensor fields on the Euclidean space $\mathbb R^n$ equipped with the standard Euclidean metric. However, these notions can be extended to general Riemannian manifolds (see \cite{scott1995Lp, bauer2010sobolev, bauer2013sobolev}).
The notion of H\"older continuity can be further extended to the manifold setting as we will explain in this section.
Throughout the article, we assume that the manifold $Q$ we work with satisfies the following properties.

\begin{assump}\label{assump:manifold}
The manifold $Q$ is smooth, connected, Hausdorff, second-countable, and paracompact. 
\end{assump}

Given a manifold $Q$ satisfying Assumption \ref{assump:manifold}, we denote the space of $(r,s)$-tensors on $Q$ by $T^{(r,s)}Q$. This is called the $(r,s)$-tensor bundle. We refer to a ``tensor field" as a section of a tensor bundle $T^{(r, s)}Q$, which is a map $\sigma : Q \rightarrow T^{(r, s)}Q$ with the property $\pi \circ \sigma = id_{Q}$, where $\pi : T^{(r, s)}Q \rightarrow Q$ is the canonical projection. We denote the space of $(r,s)$-tensor fields over $Q$ by $\Gamma(T^{(r, s)}Q)$.

\begin{definition}
For $\alpha \in (0,1)$, we say that a tensor field $K \in \Gamma(T^{(r,s)}Q)$ is (locally) $\alpha$-H\"older continuous if on every local chart, the tensor $K$ expressed in coordinates is (locally) $\alpha$-H\"older continuous in the sense defined in Section \ref{holder}.
\end{definition}

The above definition can be checked to be chart-independent, owing to the smoothness of transition maps between charts.
We denote the Fr\'echet space of $\alpha$-H\"older tensor fields by $C^\alpha(T^{(r, s)}Q)$.

\paragraph{\bf $C^k$-class of tensor fields.}
Let $k \in \mathbb{N}.$ We say that a section $K$ of $T^{(r,s)}Q$ is of differentiability class $C^k$ if it is $k$-times differentiable in all the charts of a given atlas. This is equivalent to requiring that each component $K^{i_1,\ldots,i_r}_{j_{1},\ldots,j_{s}}$ of $K$ in a chart is $k$ times differentiable. Again, one can show that this definition is independent of charts, owing to the smoothness of the transition maps. We call the set of all $(r,s)$-tensor fields on $Q$ of class $C^k$ the $C^k$-section of the tensor bundle, which we denote by $C^k(T^{(r,s)}Q)$. This definition can be easily extended to the case $k=\infty,$ in which case we employ the notation $C^\infty(T^{(r,s)}Q)$.
% As it is classically done, we denote the space of smooth vector fields ($(r,s)=(1,0)$) by $\mathfrak{X}(Q)$ and the space of smooth one-forms ($(r,s)=(0,1)$) by $\Omega^1(Q).$

% If $K$ is an $(r,s)$-tensor field of class $C^1$, then we define its covariant derivative $\nabla K$ as the $(r,s+1)$-tensor field with local coordinate expression
% \begin{align}
% \begin{split}
%     [\nabla K]_{l, j_1, \ldots, j_s}^{i_1, \ldots, i_r} := \frac{\partial K^{i_1,\ldots,i_r}_{j_1,\ldots,j_s}}{\partial x_l} &+ \Gamma^{i_1}_{ml} K^{m, i_2,\ldots,i_r}_{j_1,\ldots,j_s} + \cdots + \Gamma^{i_r}_{ml} K^{i_1,\ldots,i_{r-1}, m}_{j_1,\ldots,j_s} \\
%     &  - \Gamma^{m}_{j_1 l} K^{i_1,\ldots,i_r}_{m, j_2,\ldots,j_s} - \cdots - \Gamma^{m}_{j_s l} K^{i_1,\ldots,i_r}_{j_1,\ldots,j_{s-1}, m},
% \end{split}\label{eq:linear-connection}
% \end{align}
% where $(\Gamma^i_{jk})_{i,j,k = 1}^n$ 
% are the Christoffel symbols and summation over repeated indices is again assumed. We can further define the $k$-th covariant derivative $\nabla^k K$ for $C^k$-tensor fields $K$, which is an $(r,s+k)$-tensor field, by repeating this operation $k$-times.

We say that a tensor field $K$ is of H\"older class $C^{k+\alpha}$ if on every chart, the tensor $K$ expressed in coordinates are of class $C^{k+\alpha}$ in the sense defined in Section \ref{holder}. Again, it can be shown that this notion is chart-independent and therefore well-defined.

\subsection{Some operations defined on tensor fields} \label{sec:tensor-operations}

\paragraph{\bf Pull-back and push-forward.}
Let $\varphi : Q \rightarrow Q'$ be a smooth mapping between manifolds. The push-forward $\varphi_* X$ of a vector field $X \in \Gamma(TQ)$ with respect to $\varphi$ is a vector field on $Q'$ whose expression in local coordinates is given by
\begin{align*}
    (\varphi_* X)^i(\vec{\varphi}(\vec{x})) := X^j(\vec{x}) \frac{\partial \varphi^i}{\partial x^j} (\vec{x}).
\end{align*}
Similarly, the pull-back $\varphi^* \alpha$ of a one-form $\alpha \in \Gamma(T^*Q')$ with respect to $\varphi$ is a one-form on $Q$ whose local expression is given by
\begin{align*}
    (\varphi^* \alpha)_i(\vec{x}) := \alpha_j(\vec{\varphi(x)}) \frac{\partial \varphi^j}{\partial x^i}(\vec{x}).
\end{align*}
Furthermore, when $\varphi$ is a diffeomorphism, one can also define the pull-back of a vector field $X \in \Gamma(TQ')$ as $\varphi^* X := (\varphi^{-1})_* X \in \Gamma(TQ)$, and the push-forward of a one-form $\alpha \in T^*Q$ as $\varphi_* \alpha := (\varphi^{-1})^* \alpha \in \Gamma(T^*Q')$.
We can check that these definitions are indeed independent of charts.

This notion can also be extended to arbitrary $(r,s)$-tensor fields. Indeed, for a smooth tensor field $K,$ we define its pull-back $\varphi^* K$ with respect to a diffeomorphism $\varphi$ by
\begin{align*}
    &(\varphi^* K)(x)(\alpha^1(x),\ldots,\alpha^r(x),v_1(x),\ldots,v_s(x)) \\
    &\quad := K(\varphi(x))(\varphi_*\alpha^1(\varphi(x)),\ldots,\varphi_*\alpha^r(\varphi(x)),\varphi_*v_1(\varphi(x)),\ldots,\varphi_*v_s(\varphi(x))),
\end{align*}
for any $\alpha^1,\ldots,\alpha^r \in \Omega^1(Q),$ and $v_1,\ldots,v_s \in \mathfrak X(Q)$.
This has the coordinate expression
\begin{align*}
    (\varphi^* K)^{i_1,\ldots,i_r}_{j_1,\ldots,j_s}(x) &= T^{p_1,\ldots,p_r}_{q_1,\ldots,q_s}(\varphi(x)) \frac{\partial \psi^{i_1}}{\partial x_{p_1}}(\varphi(x)) \cdots \frac{\partial \psi^{i_r}}{\partial x_{p_r}}(\varphi(x)) \frac{\partial \varphi^{q_1}}{\partial x_{j_1}}(x) \cdots \frac{\partial \varphi^{q_s}}{\partial x_{j_s}}(x),
\end{align*}
where $\psi := \varphi^{-1}$. Similarly, we can define the push-forward of an $(r,s)$-tensor field as $\varphi_* K := (\varphi^{-1})^* K$.

\begin{remark} \label{rmk:scalar-pullback}
In the special case $(r,s) = (0,0)$ (i.e., the scalar case), the pull-back of a scalar field $f : Q' \rightarrow \mathbb{R}$ with respect to $\varphi$ is a function $\varphi^*f : Q \rightarrow \mathbb{R}$ defined by $\varphi^*f(x) = f (\varphi(x))$ for all $x \in Q$. Likewise, the push-forward of a scalar is defined by $\varphi_*f(x) = f (\varphi^{-1}(x))$, provided $\varphi$ is a diffeomorphism. Thus, the pull-back/push-forward generalise the notion of ``function evaluation" to tensor fields.
\end{remark}

\paragraph{\bf Lie derivatives of a $C^1$ tensor field.}
Given a vector field $b \in C^1(TQ)$, whose flow we denote by $(\phi_t)_{t \in [0,T]}$, we define the Lie derivative of $K \in C^1(T^{(r,s)}Q)$ with respect to $b$ as
\begin{align}
    \mathcal L_b K := \left.\frac{\diff}{\diff t}\right|_{t=0} \phi_t^* K.
\end{align}
In local coordinates, this reads
%(see \cite{bookgeo})
\begin{align} \label{Lieformula}
\begin{split}
        (\mathcal L_b K(\vec{x}))^{i_1, \ldots, i_r}_{j_1, \ldots, j_s} &=  b^l(\vec{x})\frac{\partial K^{i_1,\ldots,i_r}_{j_1, \ldots, j_s}}{\partial x_l}(\vec{x}) - K^{l, i_2,\ldots,i_r}_{j_1, \ldots, j_s}(\vec{x}) \frac{\partial b^{i_1}}{\partial x_l}(\vec{x}) - \cdots - K^{i_1,\ldots,i_{r-1}, l}_{j_1, \ldots, j_s}(\vec{x}) \frac{\partial b^{i_r}}{\partial x_l}(\vec{x}) \\
        & \quad + K^{i_1,\ldots,i_r}_{l, j_2, \ldots, j_s}(\vec{x}) \frac{\partial b^{l}}{\partial x_{j_1}}(\vec{x}) + \cdots + K^{i_1,\ldots,i_r}_{j_1, \ldots, j_{s-1}, l}(\vec{x}) \frac{\partial b^{l}}{\partial x^{j_s}}(\vec{x}),
\end{split}
\end{align}
where sum over repeated indices is implied. 

\subsection{Some concepts from stochastic analysis} \label{probabilidad}
In this subsection, we present some results from stochastic analysis and definitions that we will use. 
%In particular, we introduce a version of Fubini and dominated convergence theorems for progressively measurable processes integrated along semimartingales. 
In particular, we provide the definition of a Brownian flow of diffeomorphisms and some related results. We also introduce the classical It\^o-Wentzell formula and Kunita's geometric extensions of It\^o's formulas.

\paragraph{\bf Notations and considerations.}
In this subsection and in general along this article, we assume that we are working with a filtered probability space $(\Omega, \mathcal{F},\mathbb{P},(\mathcal{F}_t)_{t \in I})$ satisfying the usual conditions, where $I$ might be $[0,T]$ or $[0,\infty),$ together with a family $\{(S_t^i)_{t\in I}\}_{i=1}^N$ of continuous semimartingales adapted to $(\mathcal{F}_t)_{t \in I}$. When we work with product of measure spaces, we assume that they have been completed. Sometimes we might consider other families of continuous semimartingales or of Brownian motions $\{(W_t^j)_{t \in I}\}_{j=1}^M$, adapted to the same filtration. When considering a continuous semimartingale $(S_t)_{t \in I},$ we will assume without loss of generality that $S_0 =0.$ The notation $(\mathcal{F}_t)_{t \in I}$ is usually dropped and we write $(\mathcal{F}_t)_{t}$ instead. Also, for semimartingales (or stochastic processes) we normally write $S_t$ instead of $(S_t)_{t \in I}$. Given two continuous semimartingales $X_t$ and $Y_t,$ we denote their co-variation by $[X,Y]_t$ and their respective quadratic variations by $[X]_t, [Y]_t,$ $t \in I.$ The quadratic variation $[X]_t$ of a continuous semimartingale is a pathwise continuous nondecreasing process of bounded variation. 
We are interested in working with the solution $\phi_{s,t}$ of the equation \footnote{It might seem unnatural to start by stating the SDE in Stratonovich form since more regularity conditions will be needed than in It\^o form, but this is the natural setting for our equations.}
\begin{align} \label{flownotation}
    \phi_{s,t}(x) = x + \int^t_s b(r,\phi_{s,r}(x)) \, \diff r + \sum_{i=1}^N \int^t_s \xi_i(r,\phi_{s,r}(x)) \circ \diff W_r^i,\hspace{0.2cm} 0 \leq s \leq t \leq T,\quad x \in Q,
\end{align}
where $b:Q \rightarrow TQ,$ $\xi_i:Q \rightarrow TQ$ are vector fields, $W^i_t$ are Brownian motions, $i=1,\ldots,N,$ and $``\circ"$ represents Stratonovich integration. The solution to equation \eqref{flownotation} turns out to define a stochastic flow of diffeomorphisms under some assumptions on the coefficients \cite{kunita1997stochastic}. Oftentimes we consider $s=0$ in equation \eqref{flownotation}, obtaining the (uniparametric) SDE 
\begin{align} \label{flow-map}
    \phi_{t}(x) = x + \int^t_0 b(s,\phi_{s}(x)) \, \diff s + \sum_{i=1}^N \int^t_0 \xi_i(s,\phi_{s}(x)) \circ \diff W_s^i,\hspace{0.2cm} 0 \leq t \leq T,\quad x \in Q.
\end{align}
To explain exactly what we mean by ``stochastic flow of diffeomorphisms'', we introduce an important definition inspired by \cite{itowentzell1,kunita1997stochastic}.

\begin{definition}[Stochastic flow of $C^k$-diffeomorphisms] \label{stochasticflow}
Let $k \in \mathbb{N},$ $T>0$ and $(\Omega, \mathcal{F},\mathbb{P},(\mathcal{F}_t)_{t})$ be a filtered probability space satisfying the usual conditions together with a family of independent Brownian motions $\{W_t^i\}_{i=1}^N,$ $t \in [0,T]$ (which we can interpret as standard $N$-dimensional Brownian motion) adapted to $(\mathcal{F}_t)_{t},$ and denote by $(\mathcal{F}_{s,t})_{s,t}$ the completed sigma-algebra generated by $\{W^i_u - W^i_r\}_{i=1}^N$, $s \leq r \leq u \leq t,$ for each $0\leq s < t \leq T$. 
Let $k \in \mathbb{N}.$ An $\mathcal{F}_{s,t}$-measurable random field $\varphi_{s,t}: \Omega \times Q \rightarrow Q,$ $0\leq s \leq t \leq T$ solving equation $\eqref{flownotation}$ is called a stochastic flow of $C^k$-diffeomorphisms if $\mathbb{P}$-a.s. the following properties are satisfied
\begin{itemize}
    \item $\varphi_{s,r} = \varphi_{t,r} \circ \varphi_{s,t},$ for all $0\leq s \leq t \leq r \leq T$,
    \item $\varphi_{s,s}(\cdot) = id_{Q},$ for all $0 \leq s \leq T$,
    \item $\varphi_{s,t}(\cdot)$ is a $C^k$-diffeomorphism, $0\leq s \leq t \leq T.$ Moreover, on every local chart $U$ such that $\varphi_{s,t}(x) \in U$, the spatial derivatives $D^l\vec{\varphi}_{s,t}(\vec{x}),$ $D^l\vec{\varphi}_{s,t}^{-1}(\vec{x})$ are continuous in $(s,t,\vec{x})$ for all $|l| \leq k$.
\end{itemize}
\end{definition}

\begin{definition}[$C^k$-semimartingales]
Let $k \in \mathbb{N}$. A continuous local martingale $M_t(x),$ $t \in [0,T],$ $x \in Q,$ with covariance $C(t,x,y):= [M_{\cdot}(x), M_{\cdot}(y)]_t,$ $t \in [0,T],$ $x,y \in Q,$ is said to be a $C^{k}$-local martingale if (i) $M_t(\cdot)$ and $C(t,\cdot,\cdot)$ have modifications that are $C^{k}$-smooth in space, 
and (ii) on every local chart, the processes $D^l M_t(\vec{x})$ are continuous local martingales with covariance $D^l_{\vec{x}} D^l_{\vec{x}} C(t,\vec{x},\vec{y}),$ for all multi-indices with $\lvert l \rvert \leq k$.
A continuous process $S_t(x) = A_t(x) + M_t(x)$, expressed as the sum of a continuous process of bounded variation $A_t(x)$ and a continuous local martingale $M_t(x)$, is said to be a $C^{k}$-semimartingale if (i) $A_t(\cdot)$ is a continuous $C^{k}$-process such that the processes $D^l A_t(\vec{x})$ (on any local chart) have bounded variation, for all $\lvert l \rvert \leq k$, and (ii) $M_t$ is a $C^{k}$-local martingale.
\end{definition}
%For $\alpha \in (0,1),$ we define a $C^{k+\alpha}$-flow of diffeomorphisms to be a $C^k$-flow of diffeomorphisms such that $\mathbb{P}$-a.s. its $k$-th spatial derivative $D^k \varphi_{s,t}(x)$ is $\alpha$-H\"older continuous, uniformly in $(s,t)$.
%But
%\begin{align*} 
%    \phi_{r,t}(x) = x + \int_r^t b(\tau,\phi_{r,\tau}(x)) \, \diff \tau + \sum_{j=1}^N \int_r^t \xi_j(r,\phi_{r,\tau}(x)) \circ \diff B_\tau^j.
%\end{align*}

% 207
We consider the following assumptions on the drift and diffusion vector fields of the SDE \eqref{flownotation}.

\begin{assump} \label{assumpstrong}{\ } 
Let $\alpha \in (0,1)$ and $k\in \mathbb{N}$. Let $b : [0,T] \rightarrow \Gamma(TQ)$ and $\xi_i : [0,T] \rightarrow \Gamma(TQ)$ for $i=1, \ldots, N$ be time-dependent deterministic vector fields satisfying:
\begin{enumerate}
    \item $b(t, \cdot) \in C^{k+\alpha}(TQ)$ and $\xi_i(t, \cdot) \in C^{k+1+\alpha} \cap C^{2}(TQ)$ for all $t \in [0, T]$ and $i = 1, \ldots, N$.
    \item For all $x \in Q$, $\xi_i(\cdot, x) \in C^{1}([0,T], T_xQ)$ for all $i = 1, \ldots, N$ \footnote{The $C^1$ in time assumption is necessary to convert from It\^o to Stratonovich form.}. 
    \item On every local chart $(U, \phi)$ and for every compact subset $W \subset U$, we have
    $$
    \int^T_0 \left([\vec{b}(t, \cdot)]_{k+\alpha, W} + \sum_{i=1}^N [\vec{\xi}_i(t, \cdot)]_{k+1+\alpha, W}^2\right) \diff t < \infty,
    $$
    where $\vec{b}, (\vec{\xi}_i)_{i=1}^N$ are the coordinate expressions of $b, (\xi_i)_{i=1}^N$ respectively on $U$.
\end{enumerate}
% Let $b \in L^1([0,T];C_b^{k + \alpha}(T Q))$ and $\xi_i  \in  L^2([0,T];C_b^{k+1+\alpha}(TQ)) \cap C^1([0,T];C^2_b(TQ))$ \footnote{The $C^1$ in time assumption is necessary to convert from It\^o to Stratonovich form.}, $i=1,\ldots,N$ be deterministic vector fields.
% an adapted $C^k$-semimartingale \footnote{Following Kunita, a $C^k$-semimartingale is a continuous semimartingale such that its first $k$ (spatial) derivatives are continuous semimartingales.} and a $C^{k+\alpha'}$-stochastic flow of diffeomorphisms for any $0<\alpha' < \alpha.$ 
\end{assump}

Under this assumption, we have the following existence-and-uniqueness result of \eqref{flownotation}.

\begin{proposition}[{\cite[Theorem 8.3]{kunita1984stochastic}}] \label{eq:to-blow-or-no-blow}
%(also \cite[Theorem 4.8.4]{kunita1997stochastic})
Under Assumption \ref{assumpstrong}, for all $x \in Q$, there exists a stopping time $\tau(s, x) > s$ such that \eqref{flownotation} admits a unique solution $\varphi_{t,s}(x)$ for $0 \leq s \leq t < \tau(s, x)$ and $\mathbb{P}$-a.s., one of either the following two scenarios occurs:
\begin{enumerate}
    \item (Blow-up scenario) $Q$ is non-compact, $\tau(s, x) < T$ and $\lim_{t \rightarrow \tau(s, x)} \varphi_{t,s}(x) = ``\infty"$, where $``\infty"$ is the point at infinity in the one-point compactification of $Q$.
    \item (No blow-up scenario) $Q$ is either compact or non-compact, $\lim_{t \rightarrow \tau(s, x)} \varphi_{t,s}(x)$ exists and belongs to $Q$, and $\tau(s, x) = T$.
\end{enumerate}
Moreover, by \cite[Theorem 9.2]{kunita1984stochastic}, $\varphi_{t,s}$ has a modification that is a stochastic flow of $C^{k+\alpha'}$-diffeomorphisms for any $\alpha' \in (0,\alpha)$.
\end{proposition}

\begin{remark}
Hereafter, when $(s,x) \in [0,T] \times Q$ is fixed, we sometimes employ the shorthand notation $\tau$ to denote the stopping time $\tau(s,x)$.
\end{remark}

In \cite{kunita1997stochastic}, Kunita provides the following extensions of It\^o's formula to tensor fields over manifolds.
\begin{theorem} \label{ito-eq}
Let $K \in C^3(T^{(r,s)}Q)$ and $T>0.$ Let Assumption \ref{assumpstrong} hold with $k=4$ for the drift and diffusion vector fields in \eqref{flownotation} and let $s < \tau \leq T$ be the corresponding maximal existence stopping time given by Proposition \ref{eq:to-blow-or-no-blow}. Then It\^o's second formula (i.e., It\^o's formula with respect to the final time variable) for tensor fields states that $\mathbb{P}$-a.s. for $0 \leq s \leq t < \tau,$
\begin{align}
&(\phi_{s,t})^* K-K = \int_s^t (\phi_{s,r})^* (\mathcal{L}_{b} K) \, \diff r
+  \sum_{i=1}^N \int_s^t (\phi_{s,r})^* (\mathcal{L}_{\xi_i} K) \, \diff W_r^i -  \frac{1}{2} \sum_{i=1}^N \int_s^t (\phi_{s,r})^* (\mathcal{L}^2_{\xi_i} K) \, \diff r, \label{ito-second-1}\\
&(\phi_{t,s})^* K-K = -\int_s^t (\phi_{t,r})^* (\mathcal{L}_{b} K) \, \diff r
-  \sum_{i=1}^N  \int_s^t (\phi_{t,r})^* (\mathcal{L}_{\xi_i} K) \, \diff \widehat{W}^i_r + \frac{1}{2} \sum_{i=1}^N  \int_s^t (\phi_{t,r})^* (\mathcal{L}^2_{\xi_i} K) \, \diff r. \label{ito-second-2}
\end{align}
Likewise, It\^o's first formula (i.e., It\^o's formula with respect to the initial time variable) reads
\begin{align}
&(\phi_{s,t})^* K-K = \int_s^t \mathcal{L}_{b} ((\phi_{r,t})^*K) \, \diff r
+ \sum_{i=1}^N \int_s^t \mathcal{L}_{\xi_i} ((\phi_{r,t})^* K) \, \diff \widehat{W}^i_r -  \frac{1}{2}  \sum_{i=1}^N \int_s^t \mathcal{L}^2_{\xi_i} ((\phi_{r,t})^* K) \, \diff r, \label{ito-first-1}\\
&(\phi_{t,s})^* K-K = -\int_s^t \mathcal{L}_{b} ((\phi_{r,s})^*K) \, \diff r
-   \sum_{i=1}^N \int_s^t \mathcal{L}_{\xi_i} ((\phi_{r,s})^* K) \, \diff W^i_r + \frac{1}{2} \sum_{i=1}^N \int_s^t \mathcal{L}^2_{\xi_i} ((\phi_{r,s})^* K) \, \diff r, \label{ito-first-2}
\end{align}
where $\circ \diff \widehat{W}^i_r$ denotes Stratonovich integration with respect to the backward Brownian motion. 
\end{theorem}

By Remark \ref{rmk:scalar-pullback}, we can check that \eqref{ito-second-1} is a natural generalisation of the classic It\^o's formula \eqref{eq:og-ito-formula}.

% \begin{remark}
% Equivalent formulas for the push-forward of the flow $(\varphi_{s,t})_*$ and $(\varphi_{t,s})_*$ can also be straightforwardly obtained by noting that $(\varphi)_*= (\varphi^{-1})^*$.
% \end{remark}
%\begin{remark}
%As the observant reader might have noticed, the spatial regularity assumptions on the drift and diffusion vector fields in Theorem \ref{ito-eq} seem excessive. In \cite{kunita1997stochastic}, the theorem is stated in much more generality and is hence not optimised for our particular case. However, Theorem \ref{ito-eq} holds in weaker conditions and we will provide proofs thereof (see for instance Theorem \ref{thm:IW-tensor-Ito-ver}, which establishes a more general result).
%\end{remark}
The following two theorems can be derived as corollaries of the classical It\^o-Wentzell theorems in \cite{kunita1981some}. The first one is the It\^o version and the second one is its Stratonovich equivalent.
Note that these classical results are stated on the Euclidean space and not on a general manifold.

\begin{theorem}[It\^o-Wentzell formula for flows in It\^o form] \label{itoflows}
Let $f:\Omega \times [0,T] \times \mathbb{R}^n \rightarrow \mathbb{R}$ be a semimartingale satisfying 
\begin{enumerate}[(a)] % (a), (b), (c), ...
\item $\mathbb{P}$-a.s. $f$ is continuous in $(t,\vec{x}),$
\item $f(t,\cdot)$ is of class $C^2$ a.s. for every $t \in [0,T].$
\end{enumerate}
Moreover, let $S^i_t=A^i_t+M^i_t$ be continuous semimartingales expressed uniquely as the sum of a continuous process of bounded variation and a continuous local martingale, $t \in [0,T],$ $i=1,\ldots,M$. Let $f$ have the explicit form
\begin{align*}
f(t,\vec{x}) = f(0,\vec{x}) + \sum_{i=1}^M  \int_0^t g_i(s,\vec{x})\, \diff A^i_s + \sum_{i=1}^M \int_0^t g_i(s,\vec{x}) \, \diff M^i_s, \quad t \in [0,T],\quad \vec{x} \in \mathbb{R}^n,
\end{align*}
where $g_i$ are adapted processes satisfying
\begin{enumerate}[(1)]
    \item $\mathbb{P}$-a.s. $g_i$ are continuous in $(t,\vec{x}),$
    \item $g_i(t,\cdot)$ are of class $C^1$ a.s. for every $t \in [0,T]$,
\end{enumerate}
$i=1,\ldots,M$. Finally, let $\{\vec{\phi}_t\}_{t \in [0,\tau)}$ for some stopping time $0 < \tau \leq T$, be the (semimartingale) solution to the SDE \eqref{flow-map} with drift and diffusion vector fields satisfying Assumption \ref{assumpstrong} with $k=1.$ Then $\mathbb{P}$-a.s.
\begin{align*}
    & f(t,\vec{\phi}_t(\vec{x})) =  f(0,\vec{x}) + \sum_{i=1}^M \int^t_0 g_i(s,\vec{\phi}_s(\vec{x})) \, \diff A_s^i + \sum_{i=1}^M \int^t_0 g_i(s,\vec{\phi}_s(\vec{x})) \, \diff M_s^i \nonumber \\
    & \hspace{50pt} + \int^t_0 \mathcal (b \cdot \nabla f) (s,\vec{\phi}_s(\vec{x})) \, \diff s 
    + \sum_{j=1}^N \int^t_0 (\xi_j \cdot \nabla f) (s,\vec{\phi}_s(\vec{x})) \, \diff B_s^j \nonumber \\
    &+ \sum_{i=1}^M \sum_{j=1}^N \int^t_0 (\xi_j \cdot \nabla g_i)(s,\vec{\phi}_s(\vec{x})) \, \diff [M_{\cdot}^i, B_{\cdot}^j]_s + \frac12 \sum_{j=1}^N  \int^t_0 \mathcal (\xi_j \cdot \nabla (\xi_j \cdot \nabla f)) (s,\vec{\phi}_s(\vec{x})) \, \diff s, 
\end{align*}
holds for all $t \in [0,\tau)$ and $\vec{x} \in \R^n.$ 
\end{theorem}

\begin{theorem}[It\^o-Wentzell formula for flows in Stratonovich form]  \label{straflows}
Let $f:\Omega \times [0,T] \times \mathbb{R}^n \rightarrow \mathbb{R}$ be a semimartingale satisfying
\begin{enumerate}[(a)] % (a), (b), (c), ...
\item $\mathbb{P}$-a.s. $f$ is continuous in $(t,\vec{x}),$ \label{one1}
\item $f(t,\cdot)$ is of class $C^3$ a.s. for every $t \in [0,T].$ \label{one2}
\end{enumerate}
Moreover, let $S^i_t=A^i_t+M^i_t$ be continuous semimartingales expressed uniquely as the sum of a continuous process of bounded variation and a continuous local martingale, $t \in [0,T],$ $i=1,\ldots,M$. Let $f$ have the explicit form
\begin{align*}
f(t,\vec{x}) = f(0,\vec{x}) + \sum_{i=1}^M  \int_0^t g_i(s,\vec{x})\, \diff A^i_s + \sum_{i=1}^M \int_0^t g_i(s,\vec{x}) \circ \diff M^i_s, \quad t \in [0,T],\quad \vec{x} \in \mathbb{R}^n,
\end{align*}
where $g_i$ are adapted processes satisfying
\begin{enumerate}[(1)]
    \item $\mathbb{P}$-a.s. $g_i$ are continuous in $(t,\vec{x}),$ \label{two1}
    \item $g_i(t,\cdot)$ are of class $C^2$ a.s. for every $t \in [0,T]$. \label{two2}
\end{enumerate}
Finally, assume the representations
\begin{align*}
& g_i(t,\vec{x}) = g_i(0,\vec{x}) +  \sum_{j=1}^L \int_0^t a_{ij}(s,\vec{x}) \circ \diff M^{ij}_s, \quad t \in [0,T], \quad \vec{x} \in \mathbb{R}^n,
\end{align*}
where $a_{ij}$ are adapted processes satisfying
\begin{enumerate}[(i)]
    \item $\mathbb{P}$-a.s. $a_{ij}$ are continuous in $(t,\vec{x}),$ \label{three1}
    \item $a_{ij}(t,\cdot)$ are of class $C^1$ a.s. for every $t \in [0,T],$ \label{three2}
\end{enumerate}
and $M^{ij}_t$ are continuous semimartingales, $i=1,\ldots,M,$ $j=1,\ldots, L$. Finally, let $\{\vec{\phi}_t\}_{t \in [0,\tau)}$ for some stopping time $0 < \tau \leq T$ be the (semimartingale) solution to the SDE \eqref{flow-map} with coefficients satisfying Assumption \ref{assumpstrong} with $k=1.$ Then $\mathbb{P}$-a.s.
\begin{align*} 
    & f(t,\vec{\phi}_t(\vec{x})) =  f(0,\vec{x}) + \sum_{i=1}^M \int^t_0 g_i(s,\vec{\phi}_s(\vec{x})) \, \diff A_s^i +   \sum_{i=1}^M \int^t_0 g_i(s,\vec{\phi}_s(\vec{x})) \circ \diff M^i_s \\
    & + \int^t_0 \big( b \cdot \nabla f \big)(s,\vec{\phi}_s(\vec{x})) \, \diff s 
    + \sum_{j=1}^N \int^t_0 \big( \xi_j\cdot \nabla f \big)(s,\vec{\phi}_s(\vec{x})) \circ \diff B_s^j,
\end{align*}
for all $t \in [0,\tau)$ and $\vec{x} \in \mathbb{R}^n.$
\end{theorem}

\section{Kunita-It\^o-Wentzell formula for the advection of tensor-valued processes} \label{sec:main-result}
In this section, we prove a generalisation of (1) Kunita's formulas for the evolution of a deterministic tensor field under the action of a stochastic flow of diffeomorphisms (see Theorem \ref{ito-eq}), and (2) the It\^o-Wentzell formula for the evolution of a semimartingale-driven scalar field evaluated on a stochastic flow of diffeomorphisms (Theorem \ref{itoflows}). We call such generalisation the It\^o-Kunita-Wentzell (KIW) formula for the evolution of a tensor-field-valued semimartingale under the action of a stochastic flow of diffeomorphisms. This is the main result of the paper. %We start by establishing the It\^o version and then derive its Stratonovich counterpart.

\begin{theorem}[KIW formula for tensor-valued processes] \label{thm:IW-tensor-Ito-ver} Let $K: \Omega \times [0,T] \rightarrow  \Gamma(T^{(r,s)}Q)$ be a tensor-valued semimartingale satisfying 
\begin{enumerate}[(a)] % (a), (b), (c), ...
\item $\mathbb{P}$-a.s. $K$ is continuous in $(t,x),$
\item $K(t,\cdot)$ is of class $C^2$ a.s. for every $t \in [0,T].$
\end{enumerate}
Moreover, let $S^i_t = A^i_t + M^i_t$ be continuous semimartingales expressed uniquely as the sum of a continuous process of bounded variation and a continuous local martingale, $t \in [0,T],$ $i=1,\ldots,M.$ We assume that $K$ has the explicit form
\begin{align} \label{SPDE-tensor-ito}
    K(t,x) = K(0,x) + \sum_{i=1}^M \int^t_0 G_i(s,x) \, \diff A^i_s + \sum_{i=1}^M  \int^t_0 G_i(s,x) \, \diff M^i_s, \quad t \in [0,T],\quad x \in Q,
\end{align}
where $G_i$ are tensor-valued adapted processes satisfying
\begin{enumerate}
    \item $\mathbb{P}$-a.s. $G_i$ are continuous in $(t,x)$,
    \item $G_i(t,\cdot)$ are of class $C^1$ a.s. for every $t \in [0,T].$
\end{enumerate}
Let $\{\phi_t(x)\}_{t \in [0,\tau(x))}$ for some stopping time $0 < \tau(x) \leq T$ be the flow of \eqref{flow-map} with coefficients satisfying Assumption \ref{assumpstrong} with $k=1.$ Then $\mathbb{P}$-a.s., the tensor-valued process $K$ pulled back by $\phi_t$ satisfies the following relation 
\begin{align} \label{Ito-Wentzell-tensor-Ito-ver}
\begin{split}
    &\phi_t^* K(t,x) =  K(0,x) + \sum_{i=1}^M \int^t_0 \phi_s^* G_i(s,x) \, \diff A^i_s + \sum_{i=1}^M \int^t_0 \phi_s^* G_i(s,x) \,  \diff M^i_s  \\
    & \hspace{50pt}+ \int^t_0 \phi_s^* \mathcal L_b K (s,x) \,  \diff s 
    + \sum_{j=1}^N \int^t_0 \phi_s^* \mathcal L_{\xi_j} K(s,x) \, \diff B_s^j  \\
    &\hspace{20pt}+ \sum_{i=1}^M \sum_{j=1}^N \int^t_0 \phi^*_s \mathcal L_{\xi_j} G_i(s,x) \,  \diff [M^i_{\cdot}, B^j_{\cdot}]_s + \frac12 \sum_{j=1}^N  \int^t_0 \phi^*_s \mathcal L_{\xi_j} \mathcal L_{\xi_j} K(s,x) \,  \diff s,
\end{split}
\end{align}
for all $t \in [0,\tau(x))$ and $x \in Q.$ Moreover, if Assumption \ref{assumpstrong} is satisfied with $k=3,$ we have
\begin{align} \label{Ito-Wentzell-tensor-Ito-ver-push}
\begin{split}
    &(\phi_t)_* K(t,x) =  K(0,x) + \sum_{i=1}^M \int^t_0 (\phi_s)_* G_i(s,x) \, \diff A^i_s + \sum_{i=1}^M \int^t_0 (\phi_s)_* G_i(s,x) \,  \diff M^i_s  \\
    & \hspace{50pt} - \int^t_0 \mathcal L_b [(\phi_s)_* K](s,x) \,  \diff s 
    - \sum_{j=1}^N \int^t_0 \mathcal L_{\xi_j} [(\phi_s)_* K](s,x) \, \diff B_s^j  \\
    &\hspace{20pt} - \sum_{i=1}^M \sum_{j=1}^N \int^t_0 \mathcal L_{\xi_j} [(\phi_s)_* G_i](s,x) \,  \diff [M^i_{\cdot}, B^j_{\cdot}]_s + \frac12 \sum_{j=1}^N  \int^t_0 \mathcal L_{\xi_j} \mathcal L_{\xi_j} [(\phi_s)_* K](s,x) \,  \diff s.
\end{split}
\end{align}
\end{theorem}
% \textcolor{red}{A: Check that you agree with me regarding the signs in the push-forward formula please}

\begin{proof}[Proof of Theorem \ref{thm:IW-tensor-Ito-ver}]
We set $M=N=1$ without loss of generality. Fix $S \in C^\infty(T^{(s,r)} Q)$ and consider the following real-valued continuous semimartingale
\begin{align} \label{funcion}
    & F_t(x) := \left<\phi_t^*K(t,x), S(x)\right>, \quad t \in [0,T], \quad x \in Q.
\end{align}
Let $(U, \varphi)$ be a local chart and fix $x \in U \subset Q$. We have the following explicit expression for \eqref{funcion} in this chart:
% \footnote{For simplicity, we have slightly abused the notation in \eqref{funcion} to denote $\varphi(\phi_t(x))$ (i.e., the flow in local coordinates) by ``$\phi_t(\vec{x})$''.
% We have deliberately chosen to abuse the notation here as (1) the meaning of the formulas should be clear for the readers already familiar with differential geometry, (2) it avoids making the expressions unnecessarily complicated.}
\begin{align}
    F_t(\vec{x}) = K^{p_1,\ldots,p_r}_{q_1,\ldots,q_s}(t,\vec{\phi}_t(\vec{x})) S^{j_1, \ldots,j_s}_{i_1,\ldots,i_r}(\vec{x}) \prod_{\alpha = 1}^r \frac{\partial \psi_t^{i_\alpha}}{\partial y^{p_\alpha}}(\vec{\phi}_t(\vec{x})) \prod_{\beta = 1}^s \frac{\partial \phi_t^{q_\beta}}{\partial x^{j_\beta}}(\vec{x}), \quad t \in [0,\tau_U(x)),\label{eq:F}
\end{align}
where $\tau_U(x) := \inf \{t \in [0, T] : \phi_t(x) \notin U\}$ is the first exit time and we denoted $\psi_t := \phi_t^{-1}$ the inverse flow. As per convention, we have denoted using bold font $\vec{x}$, $\vec{\phi}_t(\vec{x})$ the coordinate expressions for $x, \phi_t(x)$ respectively on the chart $U$.
Using the It\^o product rule for semimartingales, we can derive the following evolution equation for $F_t$:
\begin{align*}
    \diff F_t(\vec{x}) &= S^{j_1, \ldots,j_s}_{i_1,\ldots,i_r}(\vec{x}) \prod_{\alpha = 1}^r \frac{\partial \psi_t^{i_\alpha}}{\partial y^{p_\alpha}}(\vec{\phi}_t(\vec{x})) \prod_{\beta = 1}^s \frac{\partial \phi_t^{q_\beta}}{\partial x^{j_\beta}}(\vec{x}) \,\diff K^{p_1,\ldots,p_r}_{q_1,\ldots,q_s}(t,\vec{\phi}_t(\vec{x})) \\
    &+ K^{p_1,\ldots,p_r}_{q_1,\ldots,q_s}(t,\vec{\phi}_t(\vec{x})) S^{j_1, \ldots,j_s}_{i_1,\ldots,i_r}(\vec{x}) \prod_{\beta = 1}^s \frac{\partial \phi_t^{q_\beta}}{\partial x^{j_\beta}}(\vec{x}) \,\diff \left(\prod_{\alpha = 1}^r \frac{\partial \psi_t^{i_\alpha}}{\partial y^{p_\alpha}}(\vec{\phi}_t(\vec{x}))\right) \\
    &+ K^{p_1,\ldots,p_r}_{q_1,\ldots,q_s}(t,\vec{\phi}_t(\vec{x})) S^{j_1, \ldots,j_s}_{i_1,\ldots,i_r}(\vec{x}) \prod_{\alpha = 1}^r \frac{\partial \psi_t^{i_\alpha}}{\partial y^{p_\alpha}}(\vec{\phi}_t(\vec{x})) \,\diff \left(\prod_{\beta = 1}^s \frac{\partial \phi_t^{q_\beta}}{\partial x^{j_\beta}}(\vec{x})\right) \\
    &+ S^{j_1, \ldots,j_s}_{i_1,\ldots,i_r}(\vec{x}) \prod_{\beta = 1}^s \frac{\partial \phi_t^{q_\beta}}{\partial x^{j_\beta}}(\vec{x}) \,\diff \left[K^{p_1,\ldots,p_r}_{q_1,\ldots,q_s}(\cdot,\vec{\phi}_\cdot(\vec{x})), \prod_{\alpha = 1}^r \frac{\partial \psi_\cdot^{i_\alpha}}{\partial y^{p_\alpha}}(\vec{\phi}_\cdot(\vec{x}))\right]_t \\
    &+ S^{j_1, \ldots,j_s}_{i_1,\ldots,i_r}(\vec{x}) \prod_{\alpha = 1}^r \frac{\partial \psi_t^{i_\alpha}}{\partial y^{p_\alpha}}(\vec{\phi}_t(\vec{x})) \,\diff \left[K^{p_1,\ldots,p_r}_{q_1,\ldots,q_s}(\cdot, \vec{\phi}_\cdot(\vec{x})),\prod_{\beta = 1}^s \frac{\partial \phi_\cdot^{q_\beta}}{\partial x^{j_\beta}}(\vec{x})\right]_t \\
    &+ S^{j_1, \ldots,j_s}_{i_1,\ldots,i_r}(\vec{x}) K^{p_1,\ldots,p_r}_{q_1,\ldots,q_s}(t,\vec{\phi}_t(\vec{x})) \,\diff \left[\prod_{\alpha = 1}^r \frac{\partial \psi_\cdot^{i_\alpha}}{\partial y^{p_\alpha}}(\vec{\phi}_\cdot(\vec{x})), \prod_{\beta = 1}^s \frac{\partial \phi_\cdot^{q_\beta}}{\partial x^{j_\beta}}(\vec{x})\right]_t,
\end{align*}
which should be understood as an integral expression valid for $t \in [0, \tau_U(x))$. Furthermore, noting that the local components $K^{p_1,\ldots,p_r}_{q_1,\ldots,q_s}$ on $U$ of the tensor field $K$ are scalar-valued, by applying the classic It\^o-Wentzell formula (Theorem \ref{itoflows}), for all $p_1,\ldots,p_r, q_1, \ldots, q_s \in \{1,\ldots, n\}$, we obtain 
\begin{align*}
    &\diff K^{p_1,\ldots,p_r}_{q_1,\ldots,q_s}(t,\vec{\phi}_t(\vec{x})) = G^{p_1,\ldots,p_r}_{q_1,\ldots,q_s}(t,\vec{\phi}_t(\vec{x})) \,\diff A_t + G^{p_1,\ldots,p_r}_{q_1,\ldots,q_s}(t,\vec{\phi}_t(\vec{x})) \,\diff M_t \\
    &\quad + b^k(t,\vec{\phi}_t(\vec{x})) \frac{\partial K^{p_1,\ldots,p_r}_{q_1,\ldots,q_s}}{\partial y^k} (t,\vec{\phi}_t(\vec{x})) \,\diff t 
    + \xi^k(t,\vec{\phi}_t(\vec{x})) \frac{\partial K^{p_1,\ldots,p_r}_{q_1,\ldots,q_s}}{\partial y^k} (t,\vec{\phi}_t(\vec{x})) \,\diff B_t \\
    &\quad + \xi^k(t,\vec{\phi}_t(\vec{x})) \frac{\partial G^{p_1,\ldots,p_r}_{q_1,\ldots,q_s}}{\partial y^k} (t,\vec{\phi}_t(\vec{x})) \,\diff [M_{\cdot}, B_{\cdot}]_t + \frac12  \xi^l(t,\vec{\phi}_t(\vec{x})) \frac{\partial}{\partial y^l}\left(\xi^k(t,\vec{\phi}_t(\vec{x})) \frac{\partial K^{p_1,\ldots,p_r}_{q_1,\ldots,q_s}}{\partial y^k}(t,\vec{\phi}_t(\vec{x}))\right) \, \diff t.
\end{align*}
Next, differentiating the SDE \eqref{flow-map} with respect to the spatial variable, we get
\begin{align}
    \diff D\vec{\phi}_t(\vec{x}) &= D\vec{b}(t,\vec{\phi}_t(\vec{x})) D\vec{\phi}_t(\vec{x}) \,\diff t + D\vec{\xi} (t,\vec{\phi}_t(\vec{x})) D\vec{\phi}_t(\vec{x}) \circ \diff B_t \nonumber \\
    &= [D\vec{b}(t,\vec{\phi}_t(\vec{x})) + \vec{C}_+(t,\vec{\phi}_t(\vec{x}))] D\vec{\phi}_t(\vec{x}) \,\diff t + D\vec{\xi}(t,\vec{\phi}_t(\vec{x})) D\vec{\phi}_t(\vec{x}) \,\diff B_t, \label{eq:D-phi-ito}
\end{align}
for $t \in [0,\tau_U(x))$ and $x \in U,$ where we denoted by
\begin{align*}
    (C_+)^i_j(t,\vec{x}) := \frac12 \frac{\partial}{\partial x^j}\left(\xi^k(t,\vec{x}) \frac{\partial \xi^i}{\partial x^k}(t,\vec{\phi}_t(\vec{x}))\right),
\end{align*}
the Stratonovich-to-It\^o correction term. Also observing that $D\vec{\psi}_t(\vec{\phi}_t(\vec{x})) = (D\vec{\phi}_t(\vec{x}))^{-1}$, we obtain
\begin{align}
    \diff D\vec{\psi}_t(\vec{\phi}_t(\vec{x})) &= \diff (D\vec{\phi}_t(\vec{x}))^{-1} \nonumber \\
    &= - (D\vec{\phi}_t(\vec{x}))^{-1}(\circ \diff D\vec{\phi}_t(\vec{x}))\,(D\vec{\phi}_t(\vec{x}))^{-1} \nonumber \\
    &= - (D\vec{\phi}_t(\vec{x}))^{-1}D\vec{b}(t,\vec{\phi}_t(\vec{x})) \,\diff t - (D\vec{\phi}_t(\vec{x}))^{-1}D\vec{\xi}(t,\vec{\phi}_t(\vec{x})) \circ \diff B_t \nonumber \\
    &= - D\vec{\psi}_t(\vec{\phi}_t(\vec{x}))[D\vec{b}(t,\vec{\phi}_t(\vec{x})) + \vec{C}_-(t,\vec{\phi}_t(\vec{x}))] \,\diff t - D\vec{\psi}_t(\vec{\phi}_t(\vec{x}))D\vec{\xi}(t,\vec{\phi}_t(\vec{x})) \,\diff B_t, \label{eq:D-psi-ito}
\end{align}
where
\begin{align*}
    (C_-)^i_j(t,\vec{x}) := \frac12 \left(\frac{\partial \xi^i}{\partial x^k}(t,\vec{x}) \frac{\partial \xi^k}{\partial x^j}(t,\vec{x}) - \xi^k(t,\vec{x}) \frac{\partial^2 \xi^i}{\partial x^j \partial x^k}(t,\vec{x})\right),
\end{align*}
is the corresponding Stratonovich-to-It\^o correction.
Applying the It\^o product rule for semimartingales, we obtain the following formula
\begin{align}
    &\diff \left(\prod_{\beta = 1}^s \frac{\partial \phi_t^{q_\beta}}{\partial x^{j_\beta}}(\vec{x})\right) = \sum_{a = 1}^s \left(\prod_{\beta \neq a}^s \frac{\partial \phi_t^{q_\beta}}{\partial x^{j_\beta}}(\vec{x})\right) \diff \frac{\partial \phi_t^{q_a}}{\partial x^{j_a}}(\vec{x}) + \frac12 \sum_{\substack{b,c = 1 \\ b \neq c}}^s \left(\prod_{\beta \neq b,c}^s \frac{\partial \phi_t^{q_\beta}}{\partial x^{j_\beta}}(\vec{x})\right) \diff \left[\frac{\partial \phi_\cdot^{q_b}}{\partial x^{j_b}}(\vec{x}),\frac{\partial \phi_\cdot^{q_c}}{\partial x^{j_c}}(\vec{x})\right]_t \nonumber \\
    \begin{split}
    &= \Biggl[\sum_{a = 1}^s \left(\prod_{\beta \neq a}^s \frac{\partial \phi_t^{q_\beta}}{\partial x^{j_\beta}}(\vec{x})\right) \frac{\partial}{\partial y^k}\left(b^{q_a}(t,\vec{\phi}_t(\vec{x})) + \frac12 \xi^l(t,\vec{\phi}_t(\vec{x})) \frac{\partial \xi^{q_a}}{\partial y^l}(t,\vec{\phi}_t(\vec{x}))\right)\frac{\partial \phi_t^k}{\partial x^{j_a}}(\vec{x})\\
    &\qquad + \frac12 \sum_{\substack{b,c = 1 \\ b \neq c}}^s \left(\prod_{\beta \neq b,c}^s \frac{\partial \phi_t^{q_\beta}}{\partial x^{j_\beta}}(\vec{x})\right) \frac{\partial \xi^{q_b}}{\partial y^k}(t,\vec{\phi}_t(\vec{x})) \frac{\partial \phi_t^k}{\partial x^{j_b}}(\vec{x}) \frac{\partial \xi^{q_c}}{\partial y^l}(t,\vec{\phi}_t(\vec{x})) \frac{\partial \phi_t^l}{\partial x^{j_c}}(\vec{x})\Biggr] \, \diff t \\
    &\quad + \sum_{a = 1}^s \left(\prod_{\beta \neq a}^s \frac{\partial \phi_t^{q_\beta}}{\partial x^{j_\beta}}(\vec{x})\right) \frac{\partial \xi^{q_a}}{\partial y^k}(t,\vec{\phi}_t(\vec{x}))\frac{\partial \phi_t^k}{\partial x^{j_a}}(\vec{x}) \,\diff B_t.
    \end{split} \label{eq:D-phi-prod}
\end{align}
Similarly, taking into account \eqref{eq:D-psi-ito}, we arrive at
\begin{align}
    &\diff\left(\prod_{\alpha = 1}^r \frac{\partial \psi_t^{i_\alpha}}{\partial y^{p_\alpha}}(\vec{\phi}_t(\vec{x}))\right) = -\Biggl[\sum_{a = 1}^r \left(\prod_{\alpha \neq a}^r \frac{\partial \psi_t^{i_\alpha}}{\partial y^{p_\alpha}}(\vec{\phi}_t(\vec{x}))\right) \frac{\partial \psi_t^{i_a}}{\partial y^k}(\vec{\phi}_t(\vec{x})) \nonumber \\
    &\qquad \times \left(\frac{\partial b^k}{\partial y^{p_a}}(t,\vec{\phi}_t(\vec{x})) + \frac12 
    \left(\frac{\partial \xi^k}{\partial y^l}(t,\vec{\phi}_t(\vec{x})) \frac{\partial \xi^l}{\partial x^{p_a}}(t,\vec{\phi}_t(\vec{x})) - \xi^l(t,\vec{\phi}_t(\vec{x})) \frac{\partial^2 \xi^k}{\partial x^l \partial x^{p_a}}(t,\vec{\phi}_t(\vec{x}))\right)
    \right) \nonumber \\
    &\qquad - \frac12 \sum_{\substack{b,c = 1 \\ b \neq c}}^r \left(\prod_{\alpha \neq b,c}^r \frac{\partial \psi_t^{i_\alpha}}{\partial y^{p_\alpha}}(\vec{\phi}_t(\vec{x}))\right) \frac{\partial \psi_t^{i_b}}{\partial y^k}(\vec{\phi}_t(\vec{x})) \frac{\partial \xi^k}{\partial y^{p_b}}(t,\vec{\phi}_t(\vec{x})) \frac{\partial \psi_t^{i_c}}{\partial y^l}(\vec{\phi}_t(\vec{x})) \frac{\partial \xi^l}{\partial y^{p_c}}(t,\vec{\phi}_t(\vec{x}))\Biggr] \, \diff t \nonumber \\
    &\quad - \sum_{a = 1}^r \left(\prod_{\alpha \neq a}^r \frac{\partial \psi_t^{i_\alpha}}{\partial y^{p_\alpha}}(\vec{\phi}_t(\vec{x}))\right) \frac{\partial \psi_t^{i_a}}{\partial y^k}(\vec{\phi}_t(\vec{x})) \frac{\partial \xi^k}{\partial y^{p_a}}(t,\vec{\phi}_t(\vec{x})) \,\diff B_t \label{D-psi-split},
\end{align}
for $t \in [0,\tau_U(x))$. Combining all the expressions obtained above, one can check that for the value $x \in U$ fixed at the beginning of the proof, $F_t(\vec{x})$ can be expressed as
\begin{align} \label{F-eq-tensor}
\begin{split}
F_t(\vec{x}) - F_0(\vec{x}) & = \int^t_0 \widehat{G}^{1}_s(\vec{x}) \,\diff s +  \int^t_0 \widehat{G}^{2}_s(\vec{x}) \,\diff [M_{\cdot},B_{\cdot}]_s \\
& + \int^t_0 \widehat{G}^{3}_s(\vec{x}) \,\diff A_s + \int^t_0 \widehat{H}^{1}_s(\vec{x}) \,\diff W_s + \int^t_0 \widehat{H}^{2}_s(\vec{x}) \,\diff B_s,
\end{split}
\end{align}
for $t \in [0,\tau_U(x)),$ where
\begin{align*}
    \widehat{G}^{1}_t(\vec{x}):= & S^{j_1, \ldots,j_s}_{i_1,\ldots,i_r}(\vec{x}) \left[\prod_{\alpha = 1}^r \frac{\partial \psi_t^{i_\alpha}}{\partial y^{p_\alpha}}(\vec{\phi}_t(\vec{x})) \prod_{\beta = 1}^s \frac{\partial \phi_t^{q_\beta}}{\partial x^{j_\beta}}(\vec{x}) \,b^k_t(\vec{\phi}_t(\vec{x})) \frac{\partial K^{p_1,\ldots,p_r}_{q_1,\ldots,q_s}}{\partial y^k}(t,\vec{\phi}_t(\vec{x})) \right.\\
    & + K^{p_1,\ldots,p_r}_{q_1,\ldots,q_s}(t,\vec{\phi}_t(\vec{x})) \left(\prod_{\alpha=1}^r\frac{\partial \psi_t^{i_\alpha}}{\partial y^{p_\alpha}}(\vec{\phi}_t(\vec{x})) \sum_{b=1}^s \left(\prod_{\beta \neq b}^s \frac{\partial \phi_t^{q_\beta}}{\partial x^{j_\beta}}(\vec{x})\right) \frac{\partial b^{q_b}_t}{\partial y^k}(\vec{\phi}_t(\vec{x})) \frac{\partial \phi_t^k}{\partial x^{j_b}}(\vec{x}) \right. \\
    &\left.\left.- \prod_{\beta = 1}^s \frac{\partial \phi_t^{q_\beta}}{\partial x^{j_\beta}}(\vec{x}) \sum_{c=1}^r \left(\prod_{\alpha \neq c}^r\frac{\partial \psi_t^{i_\alpha}}{\partial y^{p_\alpha}}(\vec{\phi}_t(\vec{x}))\right)\frac{\partial \psi_t^{i_c}}{\partial y^k}(\vec{\phi}_t(\vec{x})) \frac{\partial b^k_t}{\partial y^{p_c}}(\vec{\phi}_t(\vec{x}))\right) \right]\\
    &+\frac12 S^{j_1,\ldots,j_s}_{i_1,\ldots,i_r}(\vec{x})\left[ \prod_{\alpha=1}^r \frac{\partial \psi_t^{i_\alpha}}{\partial y^{p_\alpha}}(\vec{\phi}_t(\vec{x})) \prod_{\beta=1}^s \frac{\partial \phi_t^{q_\beta}}{\partial x^{j_\beta}}(\vec{x}) \xi^l_t(\vec{\phi}_t(\vec{x}))\frac{\partial}{\partial y^l}\left(\xi^k_t(\vec{\phi}_t(\vec{x}))\frac{\partial K^{p_1,\ldots,p_r}_{q_1,\ldots,q_s}(t,\vec{\phi}_t(\vec{x}))}{\partial y^k}\right)\right.\\
    &+ K^{p_1,\ldots,p_r}_{q_1,\ldots,q_s}(t,\vec{\phi}_t(\vec{x}))\left(\prod_{\beta=1}^s \frac{\partial \phi_t^{q_\beta}}{\partial x^{j_\beta}}(\vec{x})\sum_{b=1}^r\left(\prod_{\alpha \neq b}^r \frac{\partial \psi_t^{i_\alpha}}{\partial y^{p_\alpha}}(\vec{\phi}_t(\vec{x}))\right) \frac{\partial \psi_t^{i_b}}{\partial y^k}(\vec{\phi}_t(\vec{x})) \right. \\
    &\quad \times \left(\frac{\partial \xi^k_t}{\partial y^l}(\vec{\phi}_t(\vec{x})) \frac{\partial \xi^l_t}{\partial y^{p_b}}(\vec{\phi}_t(\vec{x})) - \xi_t^l(\vec{\phi}_t(\vec{x})) \frac{\partial^2 \xi^k_t}{\partial y^l \partial x^{p_b}}(\vec{\phi}_t(\vec{x}))\right) \\
    &+ \prod_{\alpha=1}^r \frac{\partial \psi_t^{i_\alpha}}{\partial y^{p_\alpha}}(\vec{\phi}_t(\vec{x}))\sum_{c=1}^s\left(\prod_{\beta \neq c}^s \frac{\partial \phi_t^{q_\beta}}{\partial x^{j_\beta}}(\vec{x})\right) \frac{\partial}{\partial y^k}\left(\xi^l_t(\vec{\phi}_t(\vec{x})) \frac{\partial \xi^{q_c}_t}{\partial y^l}(\vec{\phi}_t(\vec{x}))\right) \frac{\partial \phi_t^k}{\partial x^{j_c}}(\vec{\phi}_t(\vec{x})) \\
    &+ \prod_{\beta=1}^s \frac{\partial \phi_t^{q_\beta}}{\partial x^{j_\beta}}(\vec{x}) \sum_{b,c=1}^r\left(\prod_{\alpha \neq b,c}^r \frac{\partial \psi_t^{i_\alpha}}{\partial y^{p_\alpha}}(\vec{\phi}_t(\vec{x}))\right)\frac{\partial \psi_t^{i_b}}{\partial y^k}(\vec{\phi}_t(\vec{x}))\frac{\partial \xi^k_t}{\partial y^{p_b}}(\vec{\phi}_t(\vec{x}))\frac{\partial \psi_t^{i_c}}{\partial y^l}(\vec{\phi}_t(\vec{x}))\frac{\partial \xi^l_t}{\partial y^{p_c}}(\vec{\phi}_t(\vec{x})) \\
    &\left. + \prod_{\alpha=1}^r \frac{\partial \psi_t^{q_\beta}}{\partial y^{j_\beta}}(\vec{\phi}_t(\vec{x}))\sum_{b,c=1}^s\left(\prod_{\beta \neq b,c}^s \frac{\partial \phi_t^{q_\beta}}{\partial y^{j_\beta}}(\vec{x})\right)\frac{\partial \xi_t^{q_b}}{\partial y^k}(\vec{\phi}_t(\vec{x}))\frac{\partial \phi_t^k}{\partial x^{j_b}}(\vec{x}) \frac{\partial \xi_t^{q_c}}{\partial y^l}(\vec{\phi}_t(\vec{x}))\frac{\partial \phi_t^l}{\partial x^{j_c}}\right) \\
    &- 2\xi_t^k(\vec{\phi}_t(\vec{x}))\frac{\partial K^{p_1,\ldots,p_r}_{q_1,\ldots,q_s}}{\partial y^k}(t,\vec{\phi}_t(\vec{x}))\left(\prod_{\beta=1}^s \frac{\partial \phi_t^{q_\beta}}{\partial x^{j_\beta}}(\vec{x}) \sum_{b=1}^r\left(\prod_{\alpha \neq b}^r \frac{\partial \psi_t^{i_\alpha}}{\partial y^{p_\alpha}}(\vec{\phi}_t(\vec{x}))\right) \frac{\partial \psi_t^{i_b}}{\partial y^k}(\vec{\phi}_t(\vec{x}))\frac{\partial \xi_t^k}{\partial y^{p_b}}(\vec{\phi}_t(\vec{x}))\right. \\
    &\left.- \prod_{\alpha=1}^r \frac{\partial \psi_t^{i_\alpha}}{\partial y^{p_\alpha}}(\vec{\phi}_t(\vec{x}))\sum_{c=1}^s\left(\prod_{\beta \neq c}^s \frac{\partial \phi_t^{q_\beta}}{\partial x^{i_\beta}}(\vec{x})\right)\frac{\partial \xi_t^{q_c}}{\partial y^k}(\vec{\phi}_t(\vec{x})) \frac{\partial \phi_t^k}{\partial x^{j_c}}(\vec{x})\right) \\
    &-2K^{p_1,\ldots,p_r}_{q_1,\ldots,q_s}(t, \vec{\phi}_t(\vec{x}))\left(\sum_{b=1}^s \left(\prod_{\beta \neq b}^s \frac{\partial \phi_t^{q_\beta}}{\partial x^{j_\beta}}(\vec{x})\right)\frac{\partial \xi_t^{q_b}}{\partial y^k}(\vec{\phi}_t(\vec{x})) \frac{\partial \phi_t^k}{\partial x^{j_b}}(\vec{x})\right) \\
    &\qquad \qquad \left.\times \left(\sum_{c=1}^r\left(\prod_{\alpha \neq c}^r \frac{\partial \psi_t^{i_\alpha}}{\partial y^{p_\alpha}}(\vec{\phi}_t(\vec{x}))\right) \frac{\partial \psi_t^{i_c}}{\partial y^l}(\vec{\phi}_t(\vec{x})) \frac{\partial \xi_t^l}{\partial y^{p_c}}(\vec{\phi}_t(\vec{x}))\right)\right],
\end{align*}
\begin{align*}
    \widehat{G}^{2}_t(\vec{x}) &:= S^{j_1, \ldots,j_s}_{i_1,\ldots,i_r}(\vec{x}) \xi^k_t(\vec{\phi}_t(\vec{x})) \frac{\partial G^{p_1,\ldots,p_r}_{q_1,\ldots,q_s}}{\partial y^k}(t,\vec{\phi}_t(\vec{x})) \prod_{\alpha = 1}^r \frac{\partial \psi_t^{i_\alpha}}{\partial y^{p_\alpha}}(\vec{\phi}_t(\vec{x})) \prod_{\beta = 1}^s \frac{\partial \phi_t^{q_\beta}}{\partial x^{j_\beta}}(\vec{x}) \\
    &+ G^{p_1,\ldots,p_r}_{q_1,\ldots,q_s}(t,\vec{\phi}_t(\vec{x})) S^{j_1,\ldots,j_s}_{i_1,\ldots,i_r}(\vec{x}) \left(\prod_{\alpha=1}^r\frac{\partial \psi_t^{i_\alpha}}{\partial y^{p_\alpha}}(\vec{\phi}_t(\vec{x})) \sum_{l=1}^s \left(\prod_{\beta \neq l}^s \frac{\partial \phi_t^{q_\beta}}{\partial x^{j_\beta}}(\vec{x})\right) \frac{\partial \xi^{q_l}_t}{\partial y^k}(\vec{\phi}_t(\vec{x})) \frac{\partial \phi_t^k}{\partial x^{j_l}}(\vec{x}) \right. \\
    &\left. - \prod_{\beta = 1}^s \frac{\partial \phi_t^{q_\beta}}{\partial x^{j_\beta}}(\vec{x}) \sum_{m=1}^r \left(\prod_{\alpha \neq m}^r\frac{\partial \psi_t^{i_\alpha}}{\partial y^{p_\alpha}}(\vec{\phi}_t(\vec{x}))\right)\frac{\partial \psi_t^{i_m}}{\partial y^k}(\vec{\phi}_t(\vec{x})) \frac{\partial \xi^k_t}{\partial y^{p_m}}(\vec{\phi}_t(\vec{x}))\right),
\end{align*}
\begin{align*}
    \widehat{G}^{3}_t(\vec{x}) &:= G^{p_1,\ldots,p_r}_{q_1,\ldots,q_s}(t,\vec{\phi}_t(\vec{x})) S^{j_1,\ldots,j_s}_{i_1,\ldots,i_r}(\vec{x}) \prod_{\alpha=1}^r \frac{\partial \psi_t^{i_\alpha}}{\partial y^{p_\alpha}}(\vec{\phi}_t(\vec{x})) \prod_{\beta=1}^s \frac{\partial \phi_t^{q_\beta}}{\partial x^{j_\beta}}(\vec{x}), \\
    \widehat{H}^{1}_t(\vec{x}) &:= G^{p_1,\ldots,p_r}_{q_1,\ldots,q_s}(t,\vec{\phi}_t(\vec{x})) S^{j_1,\ldots,j_s}_{i_1,\ldots,i_r}(\vec{x}) \prod_{\alpha=1}^r \frac{\partial \psi_t^{i_\alpha}}{\partial y^{p_\alpha}}(\vec{\phi}_t(\vec{x})) \prod_{\beta=1}^s \frac{\partial \phi_t^{q_\beta}}{\partial x^{j_\beta}}(\vec{x}), \\
    \widehat{H}^{2}_t(\vec{x}) &:= S^{j_1, \ldots,j_s}_{i_1,\ldots,i_r}(\vec{x}) \xi^k_t(\vec{\phi}_t(\vec{x})) \frac{\partial K^{p_1,\ldots,p_r}_{q_1,\ldots,q_s}}{\partial y^k}(t,\vec{\phi}_t(\vec{x})) \prod_{\alpha = 1}^r \frac{\partial \psi_t^{i_\alpha}}{\partial y^{p_\alpha}}(\vec{\phi}_t(\vec{x})) \prod_{\beta = 1}^s \frac{\partial \phi_t^{q_\beta}}{\partial x^{j_\beta}}(\vec{x}) \\
    &+ K^{p_1,\ldots,p_r}_{q_1,\ldots,q_s}(t,\vec{\phi}_t(\vec{x})) S^{j_1,\ldots,j_s}_{i_1,\ldots,i_r}(\vec{x}) \left(\prod_{\alpha=1}^r\frac{\partial \psi_t^{i_\alpha}}{\partial y^{p_\alpha}}(\vec{\phi}_t(\vec{x})) \sum_{l=1}^s \left(\prod_{\beta \neq l}^s \frac{\partial \phi_t^{q_\beta}}{\partial x^{j_\beta}}(\vec{x})\right) \frac{\partial \xi^{q_l}_t}{\partial y^k}(\vec{\phi}_t(\vec{x})) \frac{\partial \phi_t^k}{\partial x^{j_l}}(\vec{x}) \right. \\
    &\left. - \prod_{\beta = 1}^s \frac{\partial \phi_t^{q_\beta}}{\partial x^{j_\beta}}(\vec{x}) \sum_{m=1}^r \left(\prod_{\alpha \neq m}^r\frac{\partial \psi_t^{i_\alpha}}{\partial y^{p_\alpha}}(\vec{\phi}_t(\vec{x}))\right)\frac{\partial \psi_t^{i_m}}{\partial y^k}(\vec{\phi}_t(\vec{x})) \frac{\partial \xi^k_t}{\partial y^{p_m}}(\vec{\phi}_t(\vec{x}))\right)
\end{align*}
in $[0,\tau_U(x)).$ Note that we have used the shorthand notations $b_t(x) := b(t,x)$, $\xi_t(x) := \xi(t,x)$ for simplicity.
Finally, from the explicit formula of the Lie derivative \eqref{Lieformula}, we have the local expression
\begin{align*}
    &\left.\left<\phi_t^* \L_b K(t,x),S(x)\right>\right|_{U} = S^{j_1, \ldots,j_s}_{i_1,\ldots,i_r}(\vec{x}) \left[b^k_t(\vec{\phi}_t(\vec{x})) \frac{\partial K^{p_1,\ldots,p_r}_{q_1,\ldots,q_s}}{\partial y^k}(t,\vec{\phi}_t(\vec{x})) \prod_{\alpha = 1}^r \frac{\partial \psi_t^{i_\alpha}}{\partial y^{p_\alpha}}(\vec{\phi}_t(\vec{x})) \prod_{\beta = 1}^s \frac{\partial \phi_t^{q_\beta}}{\partial x^{j_\beta}}(\vec{x}) \right.\\
    &\quad + K^{p_1,\ldots,p_r}_{q_1,\ldots,q_s}(t,\vec{\phi}_t(\vec{x})) \left(\prod_{\alpha=1}^r\frac{\partial \psi_t^{i_\alpha}}{\partial y^{p_\alpha}}(\vec{\phi}_t(\vec{x})) \sum_{l=1}^s \left(\prod_{\beta \neq l}^s \frac{\partial \phi_t^{q_\beta}}{\partial x^{j_\beta}}(\vec{x})\right) \frac{\partial b^{q_l}_t}{\partial y^k}(\vec{\phi}_t(\vec{x})) \frac{\partial \phi_t^k}{\partial x^{j_l}}(\vec{x}) \right. \\
    &\left.\left. \quad - \prod_{\beta = 1}^s \frac{\partial \phi_t^{q_\beta}}{\partial x^{j_\beta}}(\vec{x}) \sum_{m=1}^r \left(\prod_{\alpha \neq m}^r\frac{\partial \psi_t^{i_\alpha}}{\partial y^{p_\alpha}}(\vec{\phi}_t(\vec{x}))\right)\frac{\partial \psi_t^{i_m}}{\partial y^k}(\vec{\phi}_t(\vec{x})) \frac{\partial b^k_t}{\partial y^{p_m}}(\vec{\phi}_t(\vec{x}))\right)\right]
\end{align*}
for $x \in U$, $t \in [0,\tau_U(x)),$ and by direct calculation, we also obtain the following expression
\begin{align*}
    &\left.\left<\phi_t^* \L_\xi \L_\xi K(t,x),S(x)\right>\right|_U \\
    &= S^{j_1,\ldots,j_s}_{i_1,\ldots,i_r}(\vec{x})\left[\xi^k_t(\vec{\phi}_t(\vec{x})) \frac{\partial}{\partial y^l}\left(\xi^k_t(\vec{\phi}_t(\vec{x}))\frac{\partial K^{p_1,\ldots,p_r}_{q_1,\ldots,q_s}}{\partial y^k}(t,\vec{\phi}_t(\vec{x}))\right) \prod_{\alpha = 1}^r \frac{\partial \psi_t^{i_\alpha}}{\partial y^{p_\alpha}}(\vec{\phi}_t(\vec{x})) \prod_{\beta=1}^s \frac{\partial \phi_t^{q_\beta}}{\partial x^{j_\beta}}(\vec{x})\right. \\
    &-K^{p_1,\ldots,p_r}_{q_1,\ldots,q_s}(t,\vec{\phi}_t(\vec{x}))\left(\prod_{\beta=1}^s \frac{\partial \phi_t^{q_\beta}}{\partial x^{j_\beta}}(\vec{x}) \sum_{b=1}^r\xi^l_t(\vec{\phi}_t(\vec{x}))\frac{\partial^2 \xi^k_t}{\partial y^l \partial y^{p_b}}(\vec{\phi}_t(\vec{x}))\frac{\partial \psi_t^{i_b}}{\partial y^k}(\vec{\phi}_t(\vec{x}))\prod_{\alpha \neq b} \frac{\partial \psi_t^{i_\alpha}}{\partial y^{p_\alpha}}(\vec{\phi}_t(\vec{x}))\right. \\
    &\left.- \prod_{\alpha=1}^r \frac{\partial \psi^{i_\alpha}_t}{\partial y^{p_\alpha}}(\vec{\phi}_t(\vec{x})) \sum_{c=1}^s \xi^l_t(\vec{\phi}_t(\vec{x}))\frac{\partial^2 \xi_t^{q_c}}{\partial y^k \partial y^l}(\vec{\phi}_t(\vec{x})) \frac{\partial \phi_t^k}{\partial x^{j_c}}(\vec{x}) \prod_{\beta \neq c}^s \frac{\partial \phi_t^{q_\beta}}{\partial x^{j_\beta}}(\vec{x}) \right) \\
    &-2\xi^l_t(\vec{\phi}_t(\vec{x}))\frac{\partial K^{p_1,\ldots,p_r}_{q_1,\ldots,q_s}}{\partial y^l}(t,\vec{\phi}_t(\vec{x}))\left(\prod_{\beta=1}^s \frac{\partial \phi_t^{q_\beta}}{\partial x^{j_\beta}}(\vec{x}) \sum_{b=1}^r \frac{\partial \xi_t^k}{\partial y^{p_b}}(\vec{\phi}_t(\vec{x}))
    \frac{\partial \psi_t^{i_b}}{\partial y^k}(\vec{\phi}_t(\vec{x}))\prod_{\alpha \neq b} \frac{\partial \psi_t^{i_\alpha}}{\partial y^{p_\alpha}}(\vec{\phi}_t(\vec{x}))\right. \\
    &-\left.\prod_{\alpha=1}^r \frac{\partial \psi_t^{i_\alpha}}{\partial y^{p_\alpha}}(\vec{\phi}_t(\vec{x})) \sum_{c=1}^s \frac{\partial \xi_t^{q_c}}{\partial y^k}(\vec{\phi}_t(\vec{x})) \frac{\partial \phi_t^k}{\partial x^{j_c}}(\vec{x}) \prod_{\beta \neq c}^s \frac{\partial \phi_t^{q_\beta}}{\partial x^{j_\beta}}(\vec{x}) \right) \\
    &+ K^{p_1,\ldots,p_r}_{q_1,\ldots,q_s}(t,\vec{\phi}_t(\vec{x}))\left(\sum_{b=1}^r \frac{\partial \xi_t^k}{\partial y^l}(\vec{\phi}_t(\vec{x})) \frac{\partial \xi_t^l}{\partial y^{p_b}}(\vec{\phi}_t(\vec{x}))\frac{\partial \psi_t^{i_b}}{\partial y^k}(\vec{\phi}_t(\vec{x})) \prod_{\alpha \neq b}^r \frac{\partial \psi_t^{i_\alpha}}{\partial y^{p_\alpha}}(\vec{\phi}_t(\vec{x})) \prod_{\beta=1}^s \frac{\partial \phi_t^{q_\beta}}{\partial x^{j_\beta}}(\vec{x})\right. \\
    &+ \sum_{c=1}^s \frac{\partial \xi^l_t}{\partial y^k}(\vec{\phi}_t(\vec{x})) \frac{\partial \xi_t^{q_c}}{\partial y^l}(\vec{\phi}_t(\vec{x})) \frac{\partial \phi_t^k}{\partial x^{j_c}}(\vec{x}) \prod_{\alpha=1}^r \frac{\partial \psi_t^{i_\alpha}}{\partial y^{p_\alpha}}(\vec{\phi}_t(\vec{x})) \prod_{\beta \neq c}^s \frac{\partial \phi_t^{q_\beta}}{\partial x^{j_\beta}}(\vec{x}) \\
    &+ \sum_{b,c = 1}^r \frac{\partial \xi_t^k}{\partial y^{p_b}}(\vec{\phi}_t(\vec{x})) \frac{\partial \xi_t^l}{\partial y^{p_c}}(\vec{\phi}_t(\vec{x}))
    \frac{\partial \psi_t^{p_b}}{\partial y^k}(\vec{\phi}_t(\vec{x})) \frac{\partial \psi_t^{p_c}}{\partial y^l}(\vec{\phi}_t(\vec{x})) \prod_{\alpha \neq b,c}^r \frac{\partial \psi_t^{i_\alpha}}{\partial y^{p_\alpha}}(\vec{\phi}_t(\vec{x})) \prod_{\beta=1}^s \frac{\partial \phi_t^{q_\beta}}{\partial x^{j_\beta}}(\vec{x}) \\
    &+ \sum_{b,c = 1}^s \frac{\partial \xi_t^{q_b}}{\partial y^k}(\vec{\phi}_t(\vec{x})) \frac{\partial \xi_t^{q_c}}{\partial y^l}(\vec{\phi}_t(\vec{x})) 
    \frac{\partial \phi_t^k}{\partial y^{q_b}}(\vec{x}) \frac{\partial \phi_t^l}{\partial y^{q_c}}(\vec{x})
    \prod_{\alpha = 1}^r \frac{\partial \psi_t^{i_\alpha}}{\partial y^{p_\alpha}}(\vec{\phi}_t(\vec{x})) \prod_{\beta \neq b,c}^s \frac{\partial \phi_t^{q_\beta}}{\partial x^{j_\beta}}(\vec{x}) \\
    &\left.\left.-2 \sum_{b=1}^r \sum_{c=1}^s \frac{\partial \xi_t^k}{\partial y^{p_b}}(\vec{\phi}_t(\vec{x})) \frac{\partial \xi_t^{q_c}}{\partial y^k}(\vec{\phi}_t(\vec{x})) \frac{\partial \psi_t^{i_b}}{\partial y^l}(\vec{\phi}_t(\vec{x})) \frac{\partial \phi_t^k}{\partial x^{j_c}}(\vec{x}) \prod_{\alpha \neq b}^r \frac{\partial \psi_t^{i_\alpha}}{\partial y^{p_\alpha}}(\vec{\phi}_t(\vec{x})) \prod_{\beta \neq c}^s \frac{\partial \phi_t^{q_\beta}}{\partial x^{j_\beta}}(\vec{x})\right)\right]
\end{align*}
for $x \in U$, $t \in [0,\tau_U(x))$. Comparing the above expressions, we can check that
\begin{align*}
&\int^t_0 \widehat{G}^{1}_s(\vec{x}) \, \diff s  =  \int^t_0 \left.\langle \phi_s^* \mathcal L_b K(s,x), S(x)\rangle \right|_U\, \diff s + \frac12 \int^t_0 \left.\langle \phi_s^* \mathcal L_\xi \mathcal L_\xi K(s,x), S(x)\rangle \right|_U \, \diff s, \\
&\int^t_0 \widehat{G}^{2}_s(\vec{x}) \, \diff [W_{\cdot},B_{\cdot}]_s = \int^t_0 \left.\langle \phi^*_s \mathcal L_\xi G(s,x), S(x)\rangle \right|_U \, \diff [W_{\cdot},B_{\cdot}]_s,\\
&\int^t_0 \widehat{G}^{3}_s(\vec{x}) \, \diff A_s = \int^t_0 \left.\langle \phi_s^* G(s,x), S(x)\rangle \right|_U \, \diff A_s \\
&\int^t_0 \widehat{H}^{1}_s(\vec{x}) \,\diff W_s = \int^t_0 \left.\langle \phi^*_s G(s,x), S(x) \rangle \right|_U\, \diff W_s, \\
&\int^t_0 \widehat{H}^{2}_s(\vec{x}) \, \diff B_s = \int^t_0 \left.\langle \phi_s^* \mathcal L_\xi K(s,x), S(x) \rangle \right|_U \, \diff B_s,
\end{align*}
for all $x \in U$ and $t \in [0,\tau_U(x)).$ Since $x\in U$ was chosen arbitrarily, the above formulas are valid for any $x \in U.$ We also notice that the resulting expressions are independent of the choice of coordinates (i.e., they can be expressed in global geometric form). Hence, given two charts $U, V \subset Q$ such that $U \cap V \neq \emptyset$, we conclude that for $x \in U \cap V,$ \eqref{Ito-Wentzell-tensor-Ito-ver} holds for all $t \in [0, \tau_U(x) \wedge \tau_V(x))$. To extend our result to the maximal time interval $[0, \tau(x))$, it suffices to apply a similar argument to that in the proof of Lemma 4.8.2 in \cite{kunita1997stochastic}, which we explain in detail in Appendix \ref{appendix:time-extension}.

Finally, the proof of \eqref{Ito-Wentzell-tensor-Ito-ver-push} follows almost identical steps by noting that
\begin{enumerate}
    \item By definition, we have $(\phi_t)_* = \psi_t^*,$ where $\psi_t := \phi_t^{-1}$ is the inverse flow.
    \item By taking into account that $\psi_t(\phi_t(y)) = y,$ for all $y \in Q$ and applying the Stratonovich chain rule, we have
    \begin{align} \label{flowinv}
    \begin{split}
        0 & = \diff [\vec{\psi}_t(\vec{\phi}_t(\vec{y}))] = \diff \vec{\psi}_t (\vec{\phi}_t(\vec{y})) + \frac{\partial \vec{\psi}_t}{\partial x^i}(\vec{\phi}_t(\vec{y})) \diff \phi_t^i(\vec{y}) \\
        &= \diff \vec{\psi}_t (\vec{\phi}_t(\vec{y})) + \frac{\partial \vec{\psi}_t}{\partial x^i}(\vec{\phi}_t(\vec{y})) b^i(\vec{\phi}_t(\vec{y})) \diff t + \frac{\partial \vec{\psi}_t}{\partial x^i}(\vec{\phi}_t(\vec{y}))  \xi^i(t, \vec{\phi}_t(\vec{y}))  \circ \diff B_t,
    \end{split} 
    \end{align}
    in local coordinates.
    Therefore setting $\vec{x} := \vec{\phi}_t(\vec{y}),$ we obtain the equation satisfied by the inverse flow in local coordinates:
    \begin{align} \label{inversaflow}
        \diff \vec{\psi}_t(\vec{x}) = -b^i(t,\vec{x}) \frac{\partial \vec{\psi}_t}{\partial x_i}(\vec{x}) \diff t - \xi^i(t,\vec{x}) \frac{\partial \psi_t}{\partial x_i}(\vec{x}) \circ \diff B_t, \quad  t \in [0,T], \quad \vec{x} \in \R^n.
    \end{align}
    \item Apply the classic It\^o-Wentzell formula (Theorem \ref{itoflows}) to evaluate $K_{i_1,\ldots,i_k}(t,\vec{\psi}_t(\vec{x}))$ in \eqref{eq:F}, using equation \eqref{inversaflow} instead of \eqref{flow-map}.
    \item Follow the remaining steps in the proof.
\end{enumerate}
\end{proof}

When enough regularity is available, a version of Theorem \ref{thm:IW-tensor-Ito-ver} in Stratonovich form can be established, which we state in the following result.

\begin{theorem}[KIW formula for tensor-valued processes: Stratonovich version] \label{thm:IW-tensor-Strat-ver}
Let $K: \Omega \times [0,T] \rightarrow  \Gamma(T^{(r,s)}Q)$ be a tensor-valued semimartingale satisfying 
\begin{enumerate}[(a)] % (a), (b), (c), ...
\item $\mathbb{P}$-a.s. $K$ is continuous in $(t,x),$ \label{ONE1}
\item $K(t,\cdot)$ is of class $C^3$ a.s. for every $t \in [0,T].$ \label{ONE2}
\end{enumerate}
Moreover, let $S^i_t = A^i_t + M^i_t$ be continuous semimartingales expressed uniquely as the sum of a continuous process of bounded variation and a continuous local martingale, $t \in [0,T],$ $i=1,\ldots,M.$ Let $K$ have the explicit form
\begin{align} \label{K-eq-strat}
    K(t,x) = K(0,x) + \sum_{i=1}^M \int^t_0 G_i(s,x) \, \diff A_s^i + \sum_{i=1}^M \int^t_0 G_i(s,x) \circ \diff M_s^i, \quad t \in [0,T],\quad x \in Q,
\end{align}
where $G_i$ are tensor-valued adapted processes satisfying
\begin{enumerate}[(1)]
    \item $\mathbb{P}$-a.s. $G_i$ are continuous in $(t,x),$ \label{TWO1}
    \item $G_i(t,\cdot),$ are of class $C^2$ a.s. for every $t \in [0,T],$ \label{TWO2}
\end{enumerate}
$i=1,\ldots,M$. Finally, assume the representation
\begin{align} \label{ultimaassumption}
& G_i(t,x) = G_i(0,x) + \sum_{k=1}^L \int^t_0 g_{ik}(s,x) \, \diff N^{ik}_s, \quad t \in [0,T],\quad x \in Q,
\end{align}
where $h_{ik}$ are tensor-valued adapted processes such that
\begin{enumerate}[(i)] 
    \item $\mathbb{P}$-a.s. $g_{ik}$ are continuous in $(t,x),$ \label{THREE1}
    \item $g_{ik}(t,\cdot)$ are of class $C^1$ a.s. for all $t \in [0,T],$ \label{THREE2}
\end{enumerate}
and $N_t^{ik}$ are continuous semimartingales, $i=1,\ldots,M,$ $k=1,\ldots,L$.
Let $\{\phi_t(x)\}_{t \in [0,\tau(x))}$ for some stopping time $0 < \tau(x) \leq T$ be the flow of \eqref{flow-map} with coefficients satisfying Assumption \ref{assumpstrong} with $k=4.$ Then $\mathbb{P}$-a.s., the following holds
\begin{align}  \label{Ito-Wentzell-one-form-Strat-ver}
    \phi_t^* K(t,x) &=  K(0,x) + \sum_{i=1}^M \int^t_0 \phi_s^* G_i(s,x) \, \diff A^i_s + \sum_{i=1}^M \int^t_0 \phi_s^* G_i(s,x) \circ \diff M_s^i \nonumber \\
    &+ \int^t_0 \phi_s^* \mathcal L_b K (s,x) \, \diff s 
    + \sum_{j=1}^N \int^t_0 \phi_s^* \mathcal L_{\xi_j} K(s,x) \circ \diff B_s^j,\quad 
\end{align}
for all $t \in [0,\tau(x))$ and $x \in Q.$ We also have the following formula for the push-forward of $K$ by the flow
\begin{align} 
    (\phi_t)_* K(t,x) &=  K(0,x) + \sum_{i=1}^M \int^t_0 (\phi_s)_* G_i(s,x) \, \diff A_s^i + \sum_{i=1}^M \int^t_0 (\phi_s)_* G_i(s,x) \circ \diff M_s^i \nonumber \\
    & - \int^t_0 \mathcal L_b [(\phi_s)_* K] (s,x) \, \diff s
    - \sum_{j=1}^N \int^t_0 \mathcal L_{\xi_j} [(\phi_s)_* K](s,x) \circ \diff B_s^j.
\end{align}
\end{theorem}
\begin{proof}[Proof of Theorem \ref{thm:IW-tensor-Strat-ver}]
For simplicity and without loss of generality, we treat the case $M=N=1.$ We first prove \eqref{Ito-Wentzell-one-form-Strat-ver}. Consider the scalar function-valued stochastic field
\begin{align} \label{truco}
F_t(x) :=  \langle K(t,x), (\phi_t)_* S(x) \rangle,\quad t \in [0,T], \quad x \in Q,
\end{align}
%where we remind that $T:\Omega \times [0,T] \rightarrow \Gamma(T^{(r,s)}(\mathbb{R}^n))$ is a $\mathbb{P}$-a.s. continuous tensor-valued semimartingale as indicated in the hypotheses, and 
where $\langle \cdot, \cdot \rangle$ represents dual pairing and $S \in C^\infty (T^{(s,r)}Q)$ is an arbitrary smooth tensor. By the Stratonovich product rule, we have
%(see \eqref{stra-product-rule}) \footnote{Observe that $\phi_t^* F_t(x) =  \langle \phi_t^* T(t,x), S(x) \rangle$ is a continuous semimartingale since $T(t,x)$ and $\phi_t$ are continuous $C^1$-semimartingales, so $F_t$ is also a continuous semimartingale.}
\begin{align} \label{product}
\diff F_t(x) =  \langle \circ \diff K(t,x), (\phi_t)_* S(x) \rangle + \langle K(t,x), \circ \diff ((\phi_t)_* S(x)) \rangle.
\end{align}
We apply Kunita's push-forward formula to the tensor $S$ (i.e., use \eqref{ito-second-1} with $s=0$ and observe that $(\phi^{-1}_t)^*=(\phi_t)_{*}$), obtaining the following expression for the term $\diff ((\phi_t)_* S)$
\begin{align*}
&(\phi_{t})_* S-S = -\int_0^t \mathcal{L}_{b} [(\phi_{s})_* S] \, \diff s
-  \int_0^t \mathcal{L}_{\xi} [(\phi_{s})_* S] \circ \diff B_s.
\end{align*}
By substituting this formula into \eqref{product} and taking into account the equation for $K$, we obtain
\begin{align*} 
& \diff F_t(x) =  \langle G(t,x), (\phi_t)_*  S(x) \rangle \, \diff A_t +  \langle G(t,x), (\phi_t)_*  S(x) \rangle \circ \diff M_t \\
& \hspace{90pt} -  \langle K(t,x), \mathcal{L}_{b} [(\phi_{t})_* S](x) \rangle \, \diff t - \langle K(t,x), \mathcal{L}_{\xi} [(\phi_{t})_* S](x) \rangle \circ \diff B_t.
\end{align*}
We note that the above equation satisfies the hypotheses of the It\^o-Wentzell formula for flows in Stratonovich form (Theorem \ref{straflows}). \footnote{We note that Theorem \ref{straflows} is stated on the Euclidean space, so in order to apply this, we work on local charts until some stopping time and extend to the maximal stopping time by patching together the local results. This is exactly the same procedure we followed in the proof of Theorem \ref{thm:IW-tensor-Ito-ver}.}
Indeed, \ref{one1} and \ref{one2} in Theorem \ref{straflows} are met due to hypotheses \ref{ONE1} and \ref{ONE2} above, the fact that $\phi_t$ is a $C^4$-flow of diffeomorphisms due Assumption \ref{assumpstrong} with $k=4,$ and the fact that $S$ is smooth, together with \eqref{product}. Moreover, \ref{two1} and \ref{two2} in Theorem \ref{straflows} are verified by taking into account the previous remarks together with hypotheses \ref{TWO1} and \ref{TWO2} above. Note that $\mathcal{L}_{b} [(\phi_{t})_* S]$ and $\mathcal{L}_{\xi} [(\phi_{t})_* S]$ are $C^2$-regular in space since $\phi_t$ is a $C^4$-flow of diffomorphisms and the Lie derivative is a first order differential operator. Finally, \ref{three1} and \ref{three2} in Theorem \ref{straflows} are met due to hypotheses \ref{THREE1} and \ref{THREE2} and previous remarks.

Hence, applying Theorem \ref{straflows} to the (real-valued) stochastic process $F_t (\phi_t(x))$, and noticing that the Lie derivative in the scalar case takes the form $\mathcal{L}_b f = b \cdot D f$, we obtain
\begin{align*} 
F_t(\phi_t(x)) = \phi_t^* F_t(x)  &= F_0(x) + \int_0^t \phi_s^* \langle G(s,x), (\phi_s)_*  S(x) \rangle \, \diff A_s  + \int_0^t \phi_s^* \langle G(s,x), (\phi_s)_*  S(x) \rangle \circ \diff M_s  \\
&\quad - \int_0^t \phi_s^* \langle K(s,x), \mathcal{L}_{b} [(\phi_{s})_*S](x) \rangle \, \diff s  - \int_0^t \phi_s^* \langle K(s,x), \mathcal{L}_{\xi}[(\phi_{s})_*S](x) \rangle \circ \diff B_s \\
&\quad + \int_0^t \phi_s^* \mathcal{L}_{b} \langle K(s,x), (\phi_{s})_*S(x) \rangle \, \diff s + \int_0^t \phi_s^* \mathcal{L}_{\xi} \langle K(s,x), (\phi_{s})_* S(x) \rangle \circ \diff B_s. \\
& = F_0(x) + \int_0^t \langle \phi_s^* G(s,x),  S(x) \rangle \, \diff A_s  + \int_0^t  \langle \phi_s^*  G(s,x),  S(x) \rangle \circ \diff M_s  \\
&\quad + \int_0^t  \mathcal{L}_{\phi_s^* b} \langle \phi_s^* K(s,x), S(x) \rangle \, \diff s - \int_0^t  \langle \phi_s^* K(s,x), \mathcal{L}_{\phi_s^*  b} S(x) \rangle \, \diff s  \\
&\quad + \int_0^t  \mathcal{L}_{\phi_s^* \xi} \langle \phi_s^* K(s,x), S(x) \rangle \circ \diff B_s  - \int_0^t  \langle \phi_s^* K(s,x), \mathcal{L}_{\phi_s^* \xi} S(x) \rangle \circ \diff B_s,
\end{align*}
where we have taken into account the fact that $\phi_t^* \mathcal{L}_X K = \mathcal{L}_{\phi_t^*X} \phi_t^* K,$ for any smooth vector field $X$ and tensor field $K$, together with the identity $\phi_t^* \circ (\phi_t)_* = id$. Also, we employ the property $\mathcal{L}_b \langle K,S \rangle =  \langle \mathcal{L}_b K,S \rangle + \langle  K, \mathcal{L}_b  S \rangle$ to conclude
\begin{align} \label{casiprueba} 
\phi_t^* F_t(x) &= F_0(x) + \int_0^t \langle \phi_s^* G(s,x),  S(x) \rangle \, \diff A_s  + \int_0^t  \langle \phi_s^*  G(s,x),  S(x) \rangle \circ \diff M_s \nonumber \\
& \quad + \int_0^t   \langle \mathcal{L}_{\phi_s^*  b} [\phi_s^* K](s,x), S(x) \rangle \, \diff s + \int_0^t  \langle \mathcal{L}_{\phi_s^* \xi} [\phi_s^* K](s,x), S(x) \rangle \circ \diff B_s \nonumber \\
&= F_0(x) + \int_0^t \langle \phi_s^* G(s,x),  S(x) \rangle \, \diff A_s  + \int_0^t  \langle \phi_s^*  G(s,x),  S(x) \rangle \circ \diff M_s \nonumber \\
& \quad + \int_0^t   \langle \phi_s^* [\mathcal{L}_{b} K](s,x), S(x) \rangle \, \diff s + \int_0^t  \langle \phi_s^* [\mathcal{L}_{\xi} K](s,x), S(x) \rangle \circ \diff B_s.
\end{align}
By applying the pull-back to \eqref{truco}, we get 
\begin{align*} 
\phi_t^* F_t(x) =  \langle \phi_t^* K(t,x), S(x) \rangle.
\end{align*}
Finally, by observing that formula \eqref{casiprueba} is valid for any smooth tensor  
$S \in C^\infty (T^{(s,r)}(Q))$, we conclude $\mathbb{P}$-a.s.,
\begin{align*} 
    & \phi_t^* K(t,x) =  K(0,x) + \int^t_0 \phi_s^* G(s,x) \, \diff A_s + \int^t_0 \phi_s^* G(s,x) \circ \diff M_s \\
    &\hspace{50pt} + \int^t_0 \phi_s^* [\mathcal L_b K] (s,x) \, \diff s
    + \int^t_0 \phi_s^* [\mathcal L_{\xi} K](s,x) \circ \diff B_s, \quad t \in [0,\tau), \quad x \in Q.
\end{align*}
% \textcolor{blue}{ST: Maybe for the inverse result, we can just refer to Appendix B, which assumes $C^3$ regularity on the flow and the proof is a lot simpler.} \textcolor{red}{A: that proof is in different regularity conditions and in Stratonovich form, so I prefer to leave it as it is. Just polish it slightly if you think there is something that could be explained better or you spot some inaccuracy}
\end{proof}

\section{Summary and discussion}\label{sec:discussion}

In this work, we established a geometric generalisation of the It\^o-Wentzell formula and Kunita's formulas \eqref{ito-second-1}--\eqref{ito-first-2} that shows how tensor-valued semimartingales over manifolds transform under the pullback/pushforward with respect to stochastic flows of SDEs. In doing so, we observe the following subtleties that are noteworthy.

\begin{itemize}
    \item The Stratonovich version of our generalised KIW formula (Theorem \ref{thm:IW-tensor-Strat-ver}) requires stronger regularity assumptions than the It\^o version (Theorem \ref{thm:IW-tensor-Ito-ver}). The stronger assumptions imposed in the former allow for a more straightforward proof.
    \item Interestingly, we note that the KIW formula \eqref{Ito-Wentzell-tensor-Ito-ver} for the pull-back of a tensor-valued process by a diffeomorphism requires less regularity assumptions on the drift and diffusion coefficients of \eqref{flow-map} than the formula for the push-forward \eqref{Ito-Wentzell-tensor-Ito-ver-push}. This is natural since in the latter formula, the Lie derivatives appear outside the push-forward, so at least three spatial derivatives of the flow are required to make sense of equation \eqref{Ito-Wentzell-tensor-Ito-ver-push}. In the former however, only one derivative is necessary as the Lie derivatives appear inside the pull-back.
    Hence, equations \eqref{Ito-Wentzell-tensor-Ito-ver} and \eqref{Ito-Wentzell-tensor-Ito-ver-push} are not strictly counterparts of one another, as one might expect. That is, they cannot be viewed as equivalent statements, where one is merely derived from the other by replacing $\phi_t$ with its inverse.
    % They have slightly different nature.
    \item Our main results (Theorems \ref{thm:IW-tensor-Ito-ver} and \ref{thm:IW-tensor-Strat-ver}) are stated up to a stopping time since on a general manifold, the flow of \eqref{flownotation} under Assumption \ref{assumpstrong} does not necessarily admit a global solution. We can however get global-in-time statements under stronger conditions. For example, when the manifold $Q$ is compact, then \eqref{flownotation} admits a global flow if the drift and diffusion vector fields satisfy Assumption \ref{assumpstrong} (see Proposition \ref{eq:to-blow-or-no-blow}). Alternatively, when $Q$ admits a Riemannian structure, then we have a well-defined notion of ``global H\"older spaces" $C_b^\alpha(T^{(r,s)} Q)$ on tensor fields over $Q$, which is a Banach space of tensor fields complete under the norm
    \begin{align}
        \|K\|_\alpha := \sup_{x, y \in Q} \frac{|K(x)-\tau_{y \rightarrow x}K(y)|}{d(x, y)^{\alpha }},
    \end{align}
    where $d(\cdot, \cdot)$ denotes the geodesic distance and $\tau_{y \rightarrow x}$ denotes parallel transport of tensors from point $y$ to $x$ along the geodesic.
    Assuming that the vector fields are globally H\"older, we can also guarantee the existence of a global-in-time solution almost surely.
\end{itemize}

We conclude by noting that the results we established in this paper could be potentially useful in the following applications, which we plan to explore in future works:
\begin{itemize}
    \item Theorems \ref{thm:IW-tensor-Ito-ver} and \ref{thm:IW-tensor-Strat-ver} can be applied in stochastic fluid dynamics to derive stochastic equations of motion on manifolds (e.g. the sphere) from first principles using conservation laws. Indeed, we can generalise the applications established in Section 4 of \cite{takao2019} on the Euclidean space to statements on general manifolds. Furthermore, conservation laws for new types of objects appearing naturally in physics (such as vector fields or one-form densities) can be included.
    % \item The topic of well-posedness/regularisation by noise started to enjoy increasing popularity in the last decade after the pioneering discovery in \cite{flandoli2010well}, which shows that stochastic noise restores uniqueness of solutions in the linear transport equation. The ultimate aim in the field is deriving a well-posedness/regularisation by noise result for nonlinear equations that could shed some light on the millennium problem of global well-posedness of the 3D Navier-Stokes equations. Theorem \ref{thm:IW-tensor-Ito-ver} finds interesting applications in the above field. Indeed, \eqref{Ito-Wentzell-tensor-Ito-ver} would be crucial for deriving uniqueness commutator estimates if the well-posedness results in \cite{flandoli2010well} were to be generalised to manifolds (instead of the Euclidean space) and to advection of general tensors (instead of advection of scalars $(r,s)=(0,0)$). For more concreteness, see the application of the classical It\^o-Wentzell formula in Section 5 in \cite{flandoli2010well} in order to derive uniqueness estimates for a commutator of the DiPerna-Lions type \cite{diperna1989ordinary}.
    \item Our results may be useful in the field of well-posedness/regularisation by noise, where the classic It\^o-Wentzell formula plays a prominent role. For example, in the work \cite{flandoli2010well}, the It\^o-Wentzell formula is used as a key tool to establish uniqueness of solutions to the transport equation on $\mathbb{R}^n$ perturbed by transport noise. Indeed, the generalised KIW formulas we derive here could be useful for proving similar statements in more general settings, such as when the transport equation is defined over manifolds, or when transport of general tensors (instead of scalars, or a tuple of scalars as considered in \cite{flandoli2010well}) are considered. For more concrete details, we refer the readers to Section 5 of \cite{flandoli2010well}.
\end{itemize}

\appendix
\section{Time extension}\label{appendix:time-extension}
In this appendix, we discuss how to extend the time window in the proof of Theorem \ref{thm:IW-tensor-Ito-ver}, so that the main result holds for all $t \in [0, \tau(x))$, where $\tau(x)$ is the maximal existence stopping time of the flow $\phi_t(x)$. To this end,  
%If $X$ is paracompact, Hausdorff and second-countable, then it is a known fact that it is metrizable [cite something], namely, that we can define a distance function $d : X \times X \rightarrow \mathbb{R}^+$ and treat $X$ as a metric space, whose metric topology coincides with the manifold topology.
we follow the argument in \cite{kunita1997stochastic}, Lemma 4.8.2, which proves the existence of a unique maximal solution of SDEs defined on manifolds under Assumption \ref{assumpstrong}. The argument can be used to extend the time window of any result involving the flow $\phi_t$, up to its maximal time of existence.

The idea of the argument is as follows. First, note that under Assumption \ref{assump:manifold}, $Q$ is metrizable by the Urysohn metrization theorem \cite{bredon2013topology}. Hence, we can define a distance function $d : Q \times Q \rightarrow \mathbb{R}^+$ and treat $Q$ as a metric space, whose metric topology coincides with the manifold topology.
Now consider three countable families $(U_l)_{l=1}^\infty, (V_l)_{l=1}^\infty, (W_l)_{l=1}^\infty$ of open neighbourhoods of $Q$, with the following properties:
\begin{enumerate}
    \item For all $l \in \mathbb{N}$, $U_l, V_l, W_l$ are open balls with a common centre and radii of $3r, 2r, r$ respectively, for arbitrary $r > 0$.
    \item $(W_l)_{l=1}^\infty$ is an open cover for $Q$. That is, for any $x \in Q,$ there exists $l \in \mathbb{N}$ such that $x \in W_l$.
\end{enumerate}
Since $(W_l)_{l=1}^\infty$ is an open cover, we set $x \in W_1$ without loss of generality. Next, we extract a sub-family of open neighbourhoods $(U_{l_m})_{m=1}^\infty, (V_{l_m})_{m=1}^\infty, (W_{l_m})_{m=1}^\infty$ and a sequence of stopping times $(\tau_m)_{m=1}^\infty$ by applying the argument presented in Algorithm \ref{alg:time-extension}, which we expressed in pseudocode format for ease of presentation.
\begin{algorithm}[h]
\begin{algorithmic}
\State Initialisation: $l_1 = 1$, $x_1 = x$, $\tau_1 = 0$.
\For{$m = 1, 2, 3, \ldots$}
    \begin{enumerate}
    \item Restrict to the local chart $U_{l_m}$ and prove ``result'' as a statement on $\mathbb{R}^n$ until the stopping time $\tau_{m+1}(x_m) := \inf \{t > 0: \phi_{\tau_m, t}(x_{m}) \notin V_{l_m}\},$ which we denote by $\tau_{m+1}$ for simplicity.
    \item Verify that ``result'' is independent of this choice of chart.
    \item Choose $m' \in \mathbb{N}$ such that $\phi_{\tau_m, \tau_{m+1}}(x_m) \in W_{m'}\backslash\bigcup_{\alpha < m'} W_\alpha$, which is always possible since $(W_l)_{l=1}^\infty$ is an open cover of $Q$.
    \item Set $l_{m+1} = m'$ and $x_{m+1} = \phi_{\tau_{m}, \tau_{m+1}}(x_m)$. Note that since $x_{m+1} \in U_{l_m} \cap U_{l_{m+1}}$, we can return to step 1.
    \end{enumerate}
\EndFor
\end{algorithmic}
\caption{Argument to extend the time window of a ``result'' involving the flow $\phi_t$.}
\label{alg:time-extension}
\end{algorithm}

In plain language, this algorithm allows us to construct a path
$U_{l_1} \cup U_{l_2} \cup \cdots$ for $\phi_t(x)$ to traverse (i.e., for each $t$ there exists $m$ such that $\phi_t(x) \in U_{l_m}$), such that on each $U_{l_m}$, the ``result'' reduces to a statement in the Euclidean space. Moreover, in \cite[Lemma 4.8.2]{kunita1997stochastic}, Kunita shows that under such construction, the maximal time of existence $\tau(x)$ for $\phi_t(x)$ can be computed as
\begin{align}
    \tau(x) = \lim_{M \rightarrow \infty}\sum_{m=0}^M \tau_{m+1}(x_{m}).
\end{align}
Hence, if we can iterate the algorithm indefinitely, then the result holds for all $t \in [0, \tau(x)).$

Coming back to our proof of Theorem \ref{thm:IW-tensor-Ito-ver}, we have already shown that in the first iteration ($m=1$), steps 1 and 2 in Algorithm \ref{alg:time-extension} hold for any $x_1 = x \in Q$, where we replace ``result'' by the statement that the KIW formula for tensor-valued processes \eqref{Ito-Wentzell-tensor-Ito-ver} holds (recall that in our proof, the chart was chosen arbitrarily, hence it also holds for the specific chart $U_{1}$ given in the algorithm). Clearly, steps 3 and 4 hold automatically, which takes us to the next iteration $m = 2$. Note that our proof of Theorem \ref{thm:IW-tensor-Ito-ver} still holds by replacing $\phi_t$ with $\phi_{s,t}$ for any $0 < s \leq t$ (i.e., the argument does not depend on the initial time of the flow) and due to the arbitrariness of the spatial variable $x_m \in Q$, we can verify steps 1 and 2 again for the updated $x_m, \tau_m$. By induction, the loop continues indefinitely. This concludes our proof that \eqref{Ito-Wentzell-tensor-Ito-ver} holds for all $t \in [0, \tau(x))$.

\section*{Acknowledgements}
{\small
ABdL has been funded by the Margarita Salas grant awarded by the Ministry of Universities granted by Order UNI/551/2021 of May 26, as well as the European Union Next Generation EU Funds. ST acknowledges the Schr\"odinger scholarship scheme for much of the funding during this work.}

\section*{Declarations}
Declaration of interests: none.

\bibliography{biblio}
\bibliographystyle{plain}

\end{document}